\documentclass[10pt,a4paper]{article}
\usepackage{graphicx}
\usepackage{amssymb,amsmath,amsthm}
\newcommand{\ie}{\emph{i.e.}}
\newcommand{\eg}{\emph{e.g.}}
\newcommand{\cf}{\emph{cf}}
\newcommand{\e}{\mathrm{e}}
\newcommand{\D}{\mathrm{d}}
\newcommand{\C}{\mathbb{C}}
\newcommand{\N}{\mathbb{N}}
\newcommand{\R}{\mathbb{R}}

\newcommand{\Hm}[1]{\leavevmode{\marginpar{\tiny%
$\hbox to 0mm{\hspace*{-0.5mm}$\leftarrow$\hss}%
\vcenter{\vrule depth 0.1mm height 0.1mm width \the\marginparwidth}%
\hbox to
0mm{\hss$\rightarrow$\hspace*{-0.5mm}}$\\\relax\raggedright #1}}}
\newcommand{\Dom}{\mathfrak{D}}
\newcommand{\supp}{\mathop{\mathrm{supp}}\nolimits}
\newcommand{\sgn}{\mathop{\mathrm{sgn}}\nolimits}
\newcommand{\esssup}{\mathop{\mathrm{ess\;\!sup}}}
\newcommand{\essinf}{\mathop{\mathrm{ess\;\!inf}}}
\newtheorem{claim}{Claim}[section]
\newtheorem{theorem}[claim]{Theorem}
\newtheorem{lemma}[claim]{Lemma}

\newtheorem{proposition}[claim]{Proposition}
\newtheorem{corollary}[claim]{Corollary}
\theoremstyle{definition}
\newtheorem{remark}[claim]{Remark}

\begin{document}


\title{\bf Spectral analysis of a quantum system
with a double line singular interaction}
\author{Sylwia Kondej$\,^a$ \ and \ David Krej\v{c}i\v{r}\'{\i}k$\,^b$}
\date{
\small \emph{
\begin{quote}
\begin{itemize}
\item[$a)$]
Institute of Physics, University of Zielona G\'ora, ul.\ Szafrana
4a, 65246 Zielona G\'ora, Poland; skondej@proton.if.uz.zgora.pl
\item[$b)$]
Department of Theoretical Physics, Nuclear Physics Institute ASCR,
25068 \v{R}e\v{z}, Czech Republic; krejcirik@ujf.cas.cz
\end{itemize}
\end{quote}
}
\medskip
4 July 2013}
\maketitle

\begin{abstract}
\noindent We consider a non-relativistic quantum particle
interacting with a singular potential supported by two parallel
straight lines in the plane. We locate the essential spectrum
under the hypothesis that the interaction asymptotically
approaches a constant value and find conditions which guarantee
either the existence of discrete eigenvalues or Hardy-type
inequalities. For a class of our models admitting a mirror
symmetry, we also establish the existence of embedded eigenvalues
and show that they turn into resonances after introducing a small
perturbation.
\end{abstract}

\vfill
\noindent
\begin{center}
\fbox{
\textbf{To appear in:} \
\emph{Publ. RIMS, Kyoto University}
}
\end{center}

\newpage
\section{Introduction}
%
%
The problem we study in this paper belongs to the line of research
often called \emph{singular perturbations of Schr\"odinger
operators}.
Let us consider a non-rel\-a\-tiv\-is\-tic quantum particle confined to a
semiconductor structure $\Sigma\subset \R^3$. Suppose the particle
has a possibility of \emph{tunnelling}, therefore the whole~$\R^3$
forms the configuration space.
On the other hand, if the device~$\Sigma$ is \emph{narrow} in a
sense we can make an idealization and assume that~$\Sigma $ is a
set of lower dimension, for example, a surface, a curve or dots
in~$\R^3$. Consequently we come to the model of quantum systems
with potential interaction supported by a \emph{null set}. The
interaction can vary on~$\Sigma $; let a function $V:\Sigma \to
\R$ denote the potential strength. Then the Hamiltonian of such a
quantum system can be symbolically written as
\begin{equation}\label{eq-formal}
  -\Delta + V \, \delta(\cdot-\Sigma)
  \,,
\end{equation}
where $-\Delta$ denotes the Laplace operator in $L^2 (\R^3)$
and~$\delta$ is the Dirac delta function. In view of singular
interactions with translational symmetry, it makes also sense to
consider one- and two-dimensional analogues of~\eqref{eq-formal}.

There are  a lot of papers devoted to an analysis of the relation
between the geometry of~$\Sigma$ and spectral properties of the
Hamiltonian with delta interactions of \emph{constant} strength,
\cf~\cite{BT, EI, EK1, EK3, EY1, EY2}; we also refer to the
monographs \cite{AGHH,Albeverio-Kurasov} with many references. The
delta-type potentials supported on infinite curves and surfaces
are particularly used for a mathematical modelling of \emph{leaky
quantum wires and graphs} \cite{Exner-leaky,BEH}.

The simplest, known model belonging to the class described
by~(\ref{eq-formal}) is given by~$N$ \emph{quantum dots in one
dimension}, \cf~\cite[Chap.~I.3, II.2]{AGHH}. Then
$\Sigma:=\{x_i\}_{i=0}^{N-1}$ and $V(x_i):=- \alpha_i \in \R$. For
one point interaction (\ie~$N=1$, $\alpha_0=:\alpha$) the system
possesses one negative eigenvalue $-\alpha^2 /4$ if, and only if,
$\alpha$~is positive.
In the case of two point interactions of \emph{equal} strength
(\ie~$N=2$, $\alpha_0=\alpha_1=:\alpha$)
separated by the distance~$2a$, there is one negative eigenvalue~$\xi_0$
if, and only if, $0<\alpha a \leq 1$ or two eigenvalues $\xi_0 < \xi_1$
if, and only if, $\alpha a > 1$.

The problem we discuss in this paper can be considered as a
\emph{generalization of the two quantum dots} in two respects.
First, our model is two-dimensional, with the set~$\Sigma$ being
one-dimensional. Second, the generalized geometry enables us to
consider potentials~$V$ of variable strength. More specifically,
we consider the singular set~$\Sigma$ composed of two infinite
lines
\begin{equation}\label{doublelines}
  \Sigma := \Sigma_- \cup \Sigma_+
  \qquad\mbox{with}\qquad
  \Sigma _{\pm} := \R \times \{\pm a\}
\end{equation}
in~$\R^2$ and
\begin{equation}\label{eq-doublelines}
  V(x) :=
\begin{cases}
  -\alpha + V_+(x)
  &\mbox{if} \quad x\in\Sigma_+
  \,,
  \\
  -\alpha + V_-(x)
  &\mbox{if} \quad x\in\Sigma_-
  \,,
\end{cases}
\end{equation}
with $V_\pm:\Sigma_\pm\to\R$ and $\alpha>0$. In the physical
setting described above, the negative part~$-\alpha$ of~$V$ models
the confinement of the particle to~$\Sigma$, while~$V_\pm$ can be
thought as a perturbation.

Our first aim is to find a \emph{self-adjoint realization} in
$L^2(\R^2)$ of~(\ref{eq-formal}) with
\eqref{doublelines}--\eqref{eq-doublelines}. We will do this by
means of a form-sum method and the resulting operator will be
called $H_{\alpha, V_+, V_- }$. Note that the delta potential in
our model does not vanish at infinity even if~$V_\pm$ do (just
because~$\alpha$ is assumed to be positive). This means that we
may expect that the essential spectrum of $H_{\alpha, V_+, V_- }$
will differ from the spectrum of the free Hamiltonian in~$\R^2$.
In our setting, the role of the \emph{unperturbed Hamiltonian} is
played by $H_{\alpha,0,0}$.

For the unperturbed Hamiltonian, the translational symmetry allows
us to decompose the operator as follows
\begin{equation}\label{eq-decomunpert0}
  H_{\alpha,0,0} \simeq (-\Delta^\R) \otimes 1 +
  1 \otimes (-\Delta_\alpha^\R)
  \qquad\mbox{on}\qquad
  L^2(\R) \otimes L^2(\R)
  \,,
\end{equation}
where $-\Delta^\R$ is the free one-dimensional Hamiltonian and
$-\Delta_\alpha^\R$ governs the aforementioned one-dimensional
system with two points interactions. The non-negative semi-axis
constitutes the spectrum of $-\Delta^\R$. On the other hand, as
was already mentioned, the spectrum of~$-\Delta_\alpha^\R$ takes
the form
$$
  \sigma(-\Delta_\alpha^\R) =
  \sigma_\mathrm{disc}(-\Delta_\alpha^\R) \cup [0,\infty)
  \,,
$$
with the negative eigenvalues
(\cf~Lemma~\ref{le-disc2points} and Figure~\ref{Fig.two})
$$
  \sigma_\mathrm{disc}(-\Delta_\alpha^\R) =
  \begin{cases}
    \{\xi_0 \}
    & \mbox{if} \quad 0< \alpha a \leq 1 \,,
    \\
    \{\xi_0,\xi_1\}
    & \mbox{if} \quad \alpha a >1 \,,
  \end{cases}
$$
(the discrete spectrum is empty in the other situations,
which is excluded here by the assumption $\alpha > 0$).
Recalling the ordering $\xi_0 < \xi_1$ (if the latter exists),
we conclude (irrespectively of the value of~$\alpha a$) with
\begin{equation}\label{spectrum0}
  \sigma(H_{\alpha,0,0})
  = \sigma_\mathrm{ess}(H_{\alpha,0,0})
  = [\xi_0,\infty)
  \,.
\end{equation}

The main results of the paper can be formulated as follows.
\medskip \\
$\bullet$ \emph{Definition of the Hamiltonian and its  resolvent.}
The definition of the Hamiltonian $H_{\alpha , V_+, V_-}$ by means
of the form-sum method is given in Section~\ref{sec-Hamiltonian}.
We also derive a Krein-like formula for the resolvent of
$H_{\alpha , V_+, V_-}$ as a useful tool for further discussion.
\medskip \\
$\bullet$ \emph{Essential spectrum.} In
Section~\ref{sec-essential} we find a weak condition preserving
the stability of the essential spectrum of $H_{\alpha, V_+, V_-}$
with respect to $H_{\alpha, 0, 0}$. We show that if $V_{\pm}$
vanish at infinity then
\begin{equation}\label{EssSpec}
  \sigma _{\mathrm{ess}}(H_{\alpha, V_+, V_-} )
  = \sigma_{\mathrm{ess}}(H_{\alpha, 0, 0})=[\xi_0, \infty )
  \,.
\end{equation}
The strategy of our proof is as follows. Using a Neumann
bracketing argument together with minimax principle, we get $
  \inf
  \sigma _{\mathrm{ess}} (H_{\alpha , V_+, V_-}) \geq \xi_0
$. The opposite inclusion $
  \sigma _{\mathrm{ess}} (H_{\alpha , V_+, V_-})
  \supseteq [\xi_0, \infty)
$ is obtained by means of the Weyl criterion adapted to
sesquilinear forms.
\medskip \\
$\bullet$ \emph{Discrete and embedded eigenvalues.} The point
spectrum is investigated in Section~\ref{sec-discrete}. We show
that the bottom of the spectrum of $H_{\alpha , V_+, V_-}$ starts
below~$\xi_0$ provided that the sum $V_++V_-$ is negative in an
integral sense. Assuming additionally that $V_\pm$ vanish at the
infinity and combining this with the previous result on the
essential spectrum, we therefore obtain a non-trivial property
\begin{equation}\label{DiscSpec}
\sigma_{\mathrm{disc}}(H_{\alpha , V_+, V_-})\neq \emptyset\,.
\end{equation}
The proof is based on finding a suitable test function in the
variational definition of the spectral threshold. We also find
conditions which guarantee the existence of embedded eigenvalues
in the system with mirror symmetry, \ie~$ V_+=V_- $.
\medskip \\
$\bullet$ \emph{Hardy inequalities.} The case of repulsive
singular potentials, \ie~$V_\pm \geq 0$, is studied  in
Section~\ref{sec-Hardy}. In order to quantify how strong the
repulsive character of singular potential is, we derive Hardy-type
inequalities
\begin{equation}\label{eq-hardy}
  H_{\alpha , V_+, V_-} -\xi_0
  \,\geq\, \varrho
\end{equation}
in the form sense, where~$\varrho:\R\to[0,\infty)$ is not
identically zero. The functional inequality~(\ref{eq-hardy}) is
useful in the study of spectral stability of~$H_{\alpha , V_+,
V_-}$; indeed, it determines a class of potentials which can be
added to our system without producing any spectrum below~$\xi_0$.
\medskip \\
$\bullet$  \emph{Resonances.} Finally, in
Section~\ref{sec-resonances} we show that breaking the mirror
symmetry  by introducing a ``perturbant" function $V_p$ on one of
the line, for example
$$
  V_+=V_0+\epsilon V_p
  \,,
  \qquad
  V_-=V_0\,,
$$
leads to resonances. These resonances are localized   near the
original embedded eigenvalues $\nu_k$ appearing when $\epsilon
=0$. Precisely, they are determined by poles of the resolvent
which take the form
$$
z_k =\nu_k + \mu_k (\epsilon)+ i\upsilon _k (\epsilon)
\,,
$$
where $\mu_k (\epsilon) =a_k \epsilon + \mathcal{O}(\epsilon^2) $
with  $a_k$ corresponding to the first order perturbation term and
$\upsilon_k (\epsilon) =b_k \mathcal{O}(\epsilon^2)$ where $b_k
<0$ establishes
 the Fermi golden rule.

\medskip

Let us conclude this introductory section by pointing out some
special notation frequently used throughout the paper. We
abbreviate $L^2 :=L^2(\R^2)$ and $L_\pm^2 :=L^2(\R\times\{\pm
a\})$. We also shortly write $W^{n,2}:=W^{n,2}(\R^2)$ for the
corresponding Sobolev spaces. The inner product and norm in $L^2$
is denoted by $(\cdot, \cdot )$ and $\|\cdot \|$, respectively.
Given a self-adjoint operator~$H$, the symbols $\sigma_{\iota}(H)$
with $\iota \in \{\mathrm{ess},  \mathrm{ac}, \mathrm{sc},
\mathrm{p}, \mathrm{disc}\}$ denote, respectively, the essential,
absolutely continuous, singularly continuous, point and discrete
spectrum of~$H$. We use the symbol~$1$ to denote identity
operators acting in various Hilbert spaces used in the paper.

\section{The Hamiltonian and its resolvent}
\label{sec-Hamiltonian}
%
Let~$V_+$ and~$V_-$ be two real-valued functions from
$L^\infty(\R)$. Given a positive number~$a$, we denote by the same
symbols the functions $V_+ \otimes 1$ and $V_- \otimes 1$ on
$\R\times\{+a\}$ and $\R\times\{-a\}$, respectively. Finally,
let~$\alpha$ be a \emph{positive} constant.

\subsection{The self-adjoint realization of the Hamiltonian}
%
Let us consider the following quadratic form
\begin{align*}
  \mathcal{E}_{\alpha,V_+,V_-}[\psi]
  & := \int_{\R^2} |\nabla\psi|^2
  + \int_{\R\times\{+a\}} (V_+ - \alpha) \, | I_+ \psi|^2
  + \int_{\R\times\{-a\}} (V_- - \alpha) \, |I_- \psi|^2
  \,,
  \\
  \Dom(\mathcal{E}_{\alpha,V_+,V_-})
  & := W^{1,2}
  \,.
\end{align*}
Here~$I_\pm $ are the trace operators
associated with the Sobolev embedding $W^{1,2} \hookrightarrow
L^2_\pm$. The corresponding sesquilinear form
will be denoted by
$\mathcal{E}_{\alpha,V_+,V_-}(\cdot,\cdot)$.

The form~$\mathcal{E}_{\alpha,V_+,V_-}$ is clearly densely defined
and symmetric. Moreover, the boundary integrals can be shown to be
a relatively bounded perturbation of the form
$\mathcal{E}_{0,0,0}$ with the relative bound less than~$1$. This
is a consequence of the boundedness of~$V_\pm$ and the following
result.
\begin{lemma}\label{Lem.Sobolev}
For every $\psi \in W^{1,2}$ and $\epsilon\in(0,1)$, we have
\begin{equation}\label{Sobolev}
  \| I_\pm \psi\|_{L_\pm^2}^2
  \leq \frac{1}{\epsilon} \, \|\psi\|_{L^2}^2
  + \epsilon \, \|\partial_2\psi\|_{L^2}^2
  \,.
\end{equation}
\end{lemma}
\begin{proof}
For every $\epsilon\in(0,1)$ and $\psi \in C_0^\infty(\R^2)$, we
have the bound
\begin{align*}
  |\psi(x_1,\pm a)|^2
  &= \int_{-\infty}^{\pm a} \partial_2|\psi|^2(x_1,x_2) \, dx_2
  = 2 \, \Re \int_{-\infty}^{\pm a} \overline{\psi(x_1,x_2)} \,
  \partial_2\psi(x_1,x_2) \,  dx_2
  \\
  & \leq \frac{1}{\epsilon} \int_\R |\psi(x_1,x_2)|^2 \, dx_2
  + \epsilon \int_\R |\partial_2\psi(x_1,x_2)|^2 \, dx_2
  \,.
\end{align*}
Integrating over~$x_1$, we therefore get~\eqref{Sobolev} for $\psi
\in C_0^\infty(\R^2)$. By density, the obtained inequality extends
to $W^{1,2}$.
\end{proof}
\begin{remark}[Relation to the generalized Kato class]
Inequality~\eqref{Sobolev} represents a quantification of the
embedding $W^{1,2} \hookrightarrow L_\pm^2$. If the support of the
singular potential has a more complicated geometry we can derive a
generalization of~(\ref{Sobolev}). Consider a Radon measure~$\mu$
on~$\R^2$ with a support on a $C^1$ curve (finite or infinite)
without self-intersections and ``near-self-intersections''
(see \cite[Sec.~4]{BEKS} for precise assumptions).
Such a measure belongs to the generalized Kato class
(\cf~\cite[Thm.~4.1]{BEKS})
and, consequently, for any $a>0$ there exists $b>0$
such that
$$
  \int_{\R^2 }|\psi|^2 \, \mathrm{d}\mu
  \leq b \, \|\psi\|_{L^2}^2
  + a\, \|\nabla \psi\|_{L^2}^2
$$
for every $\psi \in W^{1,2}$.
\end{remark}

Since~$\mathcal{E}_{0,0,0}$ is clearly closed and non-negative (it
is in fact associated with the free Hamiltonian in~$\R^2$), it
follows by the KLMN theorem \cite[Thm.~X.17]{RS2} with help of
Lemma~\ref{Lem.Sobolev} that~$\mathcal{E}_{\alpha,V_+,V_-}$ is
closed and bounded from below. Consequently, there exists a unique
bounded-from-below self-adjoint operator~$H_{\alpha,V_+,V_-}$ in
$L^2$ which is associated with~$\mathcal{E}_{\alpha,V_+,V_-}$.
(Notice that the sign of~$\alpha$ plays no role in the definition
of~$H_{\alpha,V_+,V_-}$.)

Finally, let us note that~$H_{\alpha,V_+,V_-}$ is indeed a natural
realization of the formal expression~(\ref{eq-formal}) with
\eqref{doublelines}--\eqref{eq-doublelines}. As a matter of fact,
our form~$\mathcal{E}_{\alpha,V_+,V_-}$ represents a closed
extension of the form associated with the
expression~(\ref{eq-formal}) initially considered as acting on
smooth functions rapidly decaying at the infinity of~$\R^2$.

\subsection{The transverse Hamiltonian}
Recall the decomposition~\eqref{eq-decomunpert0} for the
unperturbed Hamiltonian~$H_{\alpha,0,0}$. Here the ``transverse''
operator $-\Delta_\alpha^\R$ is associated with the form
\begin{equation*}
  \varepsilon_{\alpha}[\phi]
  := \int_{\R} |\phi'|^2
  - \alpha \, |\phi(+a)|^2
  - \alpha \, |\phi(-a)|^2
  \,, \qquad
  \Dom(\varepsilon_{\alpha})
  := W^{1,2}(\R)
  \,.
\end{equation*}
Note that the boundary values have a good meaning
in view of the embedding $W^{1,2}(\R) \hookrightarrow C^0(\R)$.
It is easy to verify
that $\Dom(-\Delta_\alpha^\R)$ consists of functions
$
  \psi \in W^{1,2}(\R) \cap W^{2,2}(\R\setminus\{-a,+a\})
$
satisfying the interface conditions
\begin{equation}\label{interface}
  \psi '(\pm a+0 )- \psi' (\pm a-0) =- \alpha \psi (\pm a)\,,
\end{equation}
where $\psi'(a\pm 0):=\lim_{\epsilon \to 0^{+}} \psi'(a\pm \epsilon)$,
and that $-\Delta_\alpha^\R\psi=-\psi''$
on $\R\setminus\{-a,+a\}$.
An alternative way of introducing $-\Delta_\alpha^\R$
is via the extension theory~\cite[Chap.~II.2]{AGHH}.

Due to \cite[Thm.~2.1.3]{AGHH}, we have $\sigma
_{\mathrm{ess}}(-\Delta_\alpha^\R)= \sigma
_{\mathrm{ac}}(-\Delta_\alpha^\R)=[0, \infty )$.
The structure of discrete spectrum depends
on the strength of~$\alpha a$.
\begin{lemma}\label{le-disc2points}
Operator $-\Delta_\alpha^\R$ has exactly two negative eigenvalues
if $\alpha a >1$ and exactly one negative eigenvalue if $0<\alpha
a \leq 1$.
\end{lemma}
\begin{proof}
Let~$\kappa>0$.
Solving the eigenvalue problem $-\psi''=-\kappa^2\psi$
on $\R\setminus\{-a,+a\}$
in terms of exponential functions decaying at infinity
and using the continuity of~$\psi$ and~\eqref{interface},
it is straightforward to see that the algebraic equation
\begin{equation}\label{eq-discev}
\frac{\alpha ^2}{4}\,\mathrm{e}^{-4\kappa a}= \left(\kappa
-\frac{\alpha }{2}\right) ^2
\end{equation}
represents a sufficient and necessary condition
for $-\kappa ^2 \in \sigma_{\mathrm{p}}(-\Delta_\alpha^\R )$.
(Alternatively, one could use directly \cite[Eq.~(2.1.33)]{AGHH}.)
Equation~\eqref{eq-discev} is equivalent to
$g_1(\kappa ) -g_2(\kappa)=0$ with
\begin{equation}\label{eq-g}
g_1(\kappa ):= \frac{\alpha }{2} \e ^{-2\kappa a }\,\quad
\mathrm{and}\quad g_2(\kappa ):= \left|\kappa -\frac{\alpha
}{2}\right|\,.
\end{equation}
Since the exponential function~$g_1$ is positive monotonously decreasing
and~$g_2$ is monotonously increasing for
$\kappa \in (\alpha/2 \,,\infty )$
with $g_2 (\alpha/2)=0$,
there exists exactly one solution $\kappa_0>\alpha/2$
of~(\ref{eq-discev}) for any $\alpha >0$.
Since~$g_1$ is strictly convex, $g_1 (0)=g_2 (0)$
and the graph of~$g_2$ is a straight line
for $\kappa \in (0,\alpha/2)$ with $g_2'(0)=-1$,
the existence of another (at most one) solution $\kappa_1 \in (0,\alpha/2)$
is determined by the derivative of~$g_1$ at~$0$.
Obviously, $g_1'(0) \geq -1$ if, and only if, $\alpha a \leq 1$.
We set $\xi_0:=-\kappa_0^2$ and $\xi_1:=-\kappa_1^2$
(if the latter exists).
\end{proof}
\begin{figure}[h!]
\begin{center}
\includegraphics[width=0.8\textwidth]{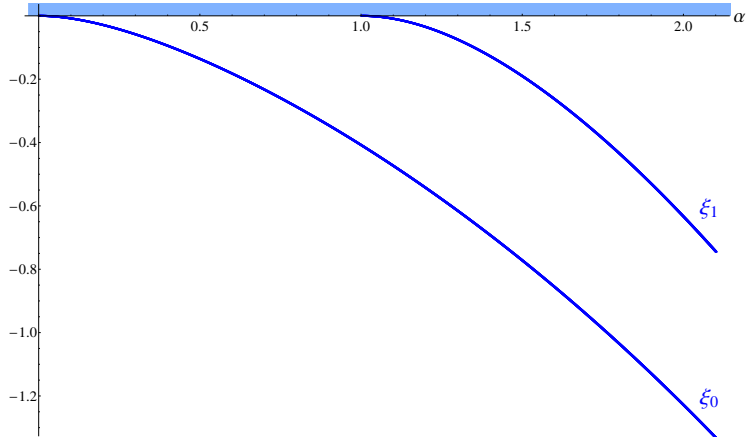}
\end{center}
\caption{Emergence of the discrete eigenvalues $\xi_0$, $\xi_1$
from the essential spectrum 
as $\alpha$ increases for $a=1$.}
\label{Fig.two}
\end{figure}

We refer to Figure~\ref{Fig.two} for the dependence of the
discrete spectrum on~$\alpha a$. The eigenfunctions corresponding
to~$\xi_0$ and~$\xi_1$ will be denoted by~$\phi_0$ and~$\phi_1$,
respectively. For notational purposes, it will become convenient
to introduce the index set
\begin{equation}\label{eq-defN}
 \mathcal{N}:=
 \begin{cases}
   \{0 \}
    & \mbox{if} \quad \alpha a \leq 1 \,,
    \\
    \{0\,,1\}
    & \mbox{if} \quad \alpha a >1 \,.
    \end{cases}
\end{equation}
Due to the symmetry of the system,
$\phi_0$~is even and $\phi_1$~is odd.
Moreover, $\phi_0$~can be chosen positive,
\cf~\cite[Thm.~2.1.3]{AGHH}.

\subsection{A lower-bound Hamiltonian}
%
In this subsection we derive an auxiliary result we shall use
several times later on. It is based on an idea used in a similar
context in~\cite{FK3}.

Taking into account the structure
of~$\mathcal{E}_{\alpha,V_+,V_-}$, let us define
\begin{equation}\label{lambda.tilde}
  \tilde\lambda(v_+,v_-) :=
  \inf_{\phi\in W^{1,2}(\R)\setminus\{0\}}
  \frac{\varepsilon_\alpha[\phi]
  + v_+ |\phi(+a)|^2 + v_- |\phi(-a)|^2}
  {\|\phi\|_{L^2(\R)}^2}
  - \xi_0
\end{equation}
for any real constants~$v_+$ and~$v_-$. The number
$\tilde\lambda(v_+,v_-)$ is the lowest eigenvalue of the
operator~\eqref{eq-formal} on $L^2(\R)$ shifted by~$\xi_0$,
subject to two point interactions of strength $v_+-\alpha$ and
$v_--\alpha$ separated by the distance~$2a$.

It is clear that~$\tilde{\lambda}(v_+,v_-)$ is a continuous and
monotonous non-decreasing function of both~$v_+$ and~$v_-$. The
symmetry relation
$\tilde{\lambda}(v_+,v_-)=\tilde{\lambda}(v_-,v_+)$ holds true.
Assume that both $v_\pm $ are non-negative. Then by the
variational definition of~$\xi_0$, we have
$\tilde{\lambda}(v_+,v_-) \geq 0$
and $\tilde{\lambda}(0,0)=0$.
For our purposes, it is important to
point out that $\tilde{\lambda}(v_+,v_-)$ is positive whenever at
least one of the arguments is. Indeed, if
$\tilde{\lambda}(v_+,v_-)=0$ and $v_+>0$ or $v_->0$, then we get
from~\eqref{lambda.tilde} that the minimum is achieved by the
eigenfunction~$\phi_0$ of $-\Delta_\alpha^\R$ corresponding
to~$\xi_0$ and that $\phi_0(a)=0$ or $\phi_0(-a)=0$. Since
$\phi_0$ is positive the latter leads to the contradiction.

If~$V_+$ and~$V_-$ are real-valued functions,
then~$\tilde{\lambda}$ gives rise to a function
$\lambda:\R^2\to\R$ via setting
\begin{equation}\label{lambda}
  \lambda(x) := \tilde\lambda\big(V_+(x_1),V_-(x_1)\big)
  \,.
\end{equation}
It follows from the properties of~$\tilde{\lambda}$ that
if~$V_\pm$ are non-negative and~$V_-$ or~$V_+$ is non-trivial
(\emph{i.e.}, non-zero on a measurable set of positive Lebesgue
measure), then~$\lambda$ is a non-trivial non-negative function.

In any case, using Fubini's theorem, we get a fundamental lower
bound for our form~$\mathcal{E}_{\alpha,V_+,V_-}$.
\begin{lemma}\label{Lem.lower}
For every $\psi \in W^{1,2}$, we have
\begin{equation}\label{pre.Hardy}
  \mathcal{E}_{\alpha,V_+,V_-}[\psi] - \xi_0 \, \|\psi\|_{L^2 }^2
  \geq \int_{\R^2} |\partial_1\psi|^2 + \int_{\R^2} \lambda \, |\psi|^2
  \,.
\end{equation}
\end{lemma}
The above result shows that $H_{\alpha,V_+,V_-}-\xi_0$ is bounded
from below by the one-dimensional Schr\"o\-dinger operator
$$
  (-\Delta^\R + \lambda) \otimes 1
  \qquad\mbox{on}\qquad
  L^2(\R) \otimes L^2(\R)
  \,.
$$

\subsection{The Krein-like resolvent formula}\label{sec-resolvent}
%
The first auxiliary step is to reconstruct the resolvent of the
unperturbed Hamiltonian, \ie\ $R_\alpha (z ):=(H_{\alpha, 0, 0}
-z)^{-1}$ for $z\in \C \setminus [\xi_0 , \infty )$. Once we have
$R_\alpha (z)$ we introduce potentials $V_\pm $ and build up the
Krein-like resolvent of $H_{\alpha , V_{+}, V_-}$ as a
perturbation of $R_\alpha (\cdot)$. For this aim we will follow
the treatment derived by Posilicano~\cite{Po}.

\subsubsection{The resolvent of the unperturbed Hamiltonian}
As above, $\xi_j$, with $j\in \mathcal{N}$, stand for the discrete
eigenvalues of $-\Delta ^\R _\alpha $ and~$\phi_j$ denote the
corresponding eigenfunctions. Recall that the essential spectrum
of $-\Delta^\R _\alpha $ is purely absolutely continuous. Let
$E_\alpha ^{\R }(\cdot )$ stand for the spectral resolution of
$-\Delta ^\R _\alpha $ corresponding to the continuous spectrum
and $P_{\alpha, j}$, with $j\in \mathcal{N}$,  denote the
eigen-projectors $P_{\alpha ,j} =\phi_j(\phi_j, \cdot )$.
Analogously, $E^\R (\cdot )$ denotes the spectral resolution of
$-\Delta^{\R}$. Using~(\ref{eq-decomunpert0}), one gets
$$
  R_\alpha (z )=\sum_{j\in \mathcal{N}}  \int_{\R_+} (\beta +\xi_j  -z
)^{-1}\mathrm{d}E^\R (\beta ) \otimes P_{\alpha, j} +
\int_{\R^2_{+}}(\beta +\beta ' -z )^{-1}\mathrm{d}E^\R (\beta )
\otimes \mathrm{d}E ^\R_\alpha (\beta ')\,,
$$
where $z\in \C \setminus [\xi_0 \,, \infty)$. This implies the
decomposition
$$
R_\alpha (z )=R_\alpha ^d (z )+R_\alpha ^c (z)\,,
$$
where $R_\alpha^\iota (z )$, with $\iota\in\{d,c\}$,
act on a separated variable function
$f(x)=f_1 (x_1) f_2 (x_2)$ as
\begin{align*} \nonumber 
R_\alpha ^d (z)f (x) &= \sum_{j\in \mathcal{N}}
\frac{1}{\sqrt{2\pi}} \int_{\R}\mathrm{d}p_1 \,
 \frac{\widehat{f}_1 (p_1) \e ^{ip_1 x_1 }}{p_1^2 +\xi_j -z}
\phi_j (x_2)  ( \phi_j , f_2)_{L^2 (\R)}\,,
\\ 
R_\alpha ^c (z)f  &= \int_{\R^2_+}\frac{1}{\beta +\beta' -z
}\mathrm{d} E^{\R} (\beta )f_1\otimes \mathrm{d}E_\alpha ^{\R
}(\beta ')f_2\,,
\end{align*}
with $\widehat{f_1}$~denoting the Fourier transform of~$f_1$. Set
$\tau _j^2 (z):=z- \xi_j $. In the following we assume that $z$ is
taken from the first  sheet of the domain of function $z\mapsto
\tau_j (z)$, \ie\ $\Im \tau_j (z)>0$.

Using the standard representation of the Green function of the
one-di\-men\-sion\-al Laplace operator
$$
(-\Delta^{\R}-k^2)^{-1}(x_1,y_1) =  \frac{1}{2\pi}\int_\R \frac{\e
^{ip (x_1-y_1)}}{p^2-k^2} \, \mathrm{d}p =\frac{i}{2} \frac{\e^{ik
|x_1-y_1|}}{k}\,
$$
we conclude that $R_\alpha ^d (z):= \sum _{j\in
\mathcal{N}} R_\alpha ^{j,d} (z)$ where $R_\alpha
^{j,d} (z)$ are integral operators with the  kernels
%
\begin{align} \label{eq-Rd}
G_\alpha ^{j,d} (z; x,y )= & 
\frac{i}{2}  \frac{\e ^{i \tau _j(z) |x_1 -y_1|}}{\tau_j (z ) }
\phi_j (x_2) \overline{\phi_j (y_2)}\,.
\end{align}
%
%
Moreover
\begin{equation}\label{eq-Rc}
  R^c _\alpha (z)f = \frac{i}{2} \int_{\R^2} \int_\R
 \frac{\e^{i\tau (p, z)|x_1-y_1|}}{\tau (p, z)}
\psi (p\,, x_2)\overline{\psi (p\,, y_2)}f_1(y_1)f_2(y_2)\,
\mathrm{d}y_1\mathrm{d}y_2\mathrm{d}p\,,
\end{equation}
where $\tau ^2(p, z)=z-p^2$, $\Im \tau (p, z)>0$ and $\psi
(p\,,x_2)$ stand for the generalized eigenfunctions of $-\Delta
^\R _\alpha $ discussed in \cite[Chap.~II.2.4]{AGHH}. To be fully
specific $\psi (p\,,x_2)$  can be obtained from Eq.~(2.4.1) of
\cite[Chap.~II]{AGHH} multiplying it by the factor $\frac{1}{2\pi
}$, see also \cite[Appendix~E]{AGHH}, Eq.~(E.5).

The resolvent of $H_{\alpha, 0, 0}$ can be written also in a
Krein-like form. We start with the resolvent of the free system $
R_{0}(z)=(H_{0,0,0} -z)^{-1} $, $z\in \C\setminus [0, \infty )$.
By means of the embeddings $ I_{\pm }\,:\, W^{1,2 }
\hookrightarrow L_\pm^{2}$ and  theirs adjoints $ I^\ast _{\pm
}\,:\, L_\pm^{2} \hookrightarrow W^{-1,2 } $, we define
$$
\hat{R}_{0 , \pm}(z) := I_{\pm} R_0 (z)\,: L^2 \to L_\pm^2
\,,\quad \check{R}_{0 , \pm}(z) :=  R_0 (z) I^\ast _{\pm}\,:
L_\pm^2 \to  L^2.
$$
 Finally, the ``bilateral'' embeddings take
the forms
$$
\mathrm{R}_{0 , ij  }(z):=I _{i} R_0 (z) I^\ast _{j} \,:\, L^2
(\R\times \{j a \}) \to L^2 (\R\times \{i a \}) \,, \qquad i,j \in
\{+ \,,-\} \,.
$$
Let $\Gamma _{0,\alpha }(z)$ denote the operator-valued matrix
acting in $L^2_+\oplus L^2_-$ and taking the form
$$
  \Gamma _{0,\alpha }(z ) :=
\begin{pmatrix}
  -\alpha^{-1}  \,,
  & 0
  \\
  0  \,,
  & -\alpha^{-1}
\end{pmatrix}
+
\begin{pmatrix}
 \mathrm{R}_{0 ,++}(z)  \,,
  & \mathrm{R}_{0 , +-}(z)
  \\
  \mathrm{R}_{0 ,-+}(z)  \,,
  & \mathrm{R}_{0 , --}(z)
\end{pmatrix}\,.
$$

Note that the continuity of $I_i \,:\,W^{1,2} \hookrightarrow
L^2_{i}$, $i \in\{ +\,,-\}$ implies the continuity of the adjoint
embedding $ I_i^\ast \, :\,L^2_{i}\hookrightarrow W^{-1,2}$.
Therefore
\begin{equation}\label{eq-aux1}
  \mathrm{Ran} \, \check{R}_{0,i } (z) \subset W^{1,2}\,.
\end{equation}
Moreover, since $\mathrm{Ran}\,I_i ^\ast  \cap L^2 =\{0\}$, we
obtain
\begin{equation}\label{eq-aux2}
\mathrm{Ran} \, \check{R}_{0,i } (z) \cap W^{2,2} =\{0\}\,.
\end{equation}
\begin{theorem} \label{th-resolvent0}
Let $z\in \C \setminus [\xi_0 , \infty )$ be such that the
operator $\Gamma _{0,\alpha }(z)$ is invertible with bounded
inverse. Then the resolvent $R_\alpha (z)$ is given by
\begin{equation}\label{eq-resolupert-ext}
R_\alpha(z)
  = R_{0 } (z)  - \sum_{i,j \in\{+, -\}}
  \check{R}_{0,i } (z)
\Gamma _{0,\alpha }(z)_{ij}^{-1} \hat{R}_{0, j} (z )\,,
\end{equation}
where $\Gamma _{0,\alpha }(z)_{ij}^{-1}$ are the matrix elements
of the inverse of $\Gamma_ {0,\alpha }(z)$.
\end{theorem}
\begin{proof}
In view of (\ref{eq-aux1}), operator $R_\alpha (z)$ defined by
(\ref{eq-resolupert-ext}) satisfies $\mathrm{Ran}\, R_{\alpha }
(z)\subset W^{1,2}$.
Assume $\psi \in W^{1,2}$ and $\phi \in L^2$. Using  $
\mathcal{E}_{0, 0, 0}( \psi, R_{0}(z)\phi ) -z ( \psi, R_{0 }(z)
\phi ) =(\psi, \phi)$, we get
\begin{eqnarray} \nonumber
\lefteqn{\mathcal{E}_{\alpha, 0, 0} \big( \psi, R_{\alpha }(z)\phi
\big) -z \big( \psi, R_{\alpha }(z) \phi \big)
 = (\psi , \phi )}
\\
\nonumber && - \sum_{i \in \{+, -\}}\alpha  \big(\psi , \hat{R}_{0
, i }(z)\phi\big)_{L^2 _i}- \sum_{i, j\in \{+, -\}} \big(\psi,
\Gamma _{0,\alpha }(z)_{ij}^{-1} \hat{R}_{0  ,j} (z ) \phi
\big)_{L^2 _i}
\\
\nonumber && +\sum_{i ,j , k\in \{+, -\}} \alpha \big(\psi ,
\mathrm{R}_{0, ij }(z) \Gamma _{0,\alpha } (z)_{jk}^{-1}
\hat{R}_{0  ,k} (z ) \phi \big)_{L^2 _i}\,,
\end{eqnarray}
where we use the embedding $W^{1,2}\hookrightarrow L_i^2$ for
$\psi $ in all expressions with scalar product in $L^2_i$.
Applying the identity
$$
  \big(\psi ,  \hat{R}_{0 , i
}(z)\phi\big)_{L^2 _i}= \sum_{j , k\in \{+, -\}} \big(\psi ,
\Gamma _{0, \alpha }(z)_{ik}\Gamma _{0, \alpha }(z)_{kj} ^{-1}
\hat{R}_{0, j }(z)\phi\big)_{L^2 _i}
$$
and using the explicit form of $\Gamma _{0,\alpha } (z)_{ij} $,
one obtains
\begin{equation}\label{eq-resolcal}
\mathcal{E}_{\alpha,0,0} \big( \psi, R_{\alpha }(z)\phi \big) -z
\big( \psi, R_{\alpha }(z) \phi \big) =(\psi, \phi)\,.
\end{equation}

Moreover, note that $R_{\alpha}(z)$ is invertible. Indeed,
employing (\ref{eq-resolupert-ext}) and
 (\ref{eq-aux2}) we conclude that
any $f\in \mathrm{Ran}\, R_{\alpha } (z)$ has a form of direct sum
$f=f_1 \dot{+} f_2$ where $f_1\in \mathrm{Ran}\, R_{0}(z)$ and
$f_2 \in W^{1,2} \setminus W^{2,2}$. This implies $\mathrm{Ker}\,
R_{\alpha } =\{0\}$ and together with (\ref{eq-resolcal})
completes the proof.
\end{proof}

\subsubsection{The Krein-like formula for the resolvent
of $H_{\alpha, V_+, V_-}$}
Relying on  (\ref{eq-resolupert-ext}) and (\ref{eq-aux1}), we can
determine the analogous embeddings operators as in the previous
section but now instead of $R_0 (z)$ we consider $R_\alpha (z)$.
Namely, define $ \hat{R}_{\alpha , \pm}(z):= I_\pm R_{\alpha }(z)
\,: L^2 \to L_\pm^2 \, $, $\check{R}_{\alpha  , \pm}(z):=
R_{\alpha }(z) I^\ast _\pm  \,:\,L_\pm^2\to L^2 $ and $
\mathrm{R}_{\alpha , ij }(z):=I_i R_{\alpha }(z) I^\ast _j
\,:\,
L^2 _j  \to L^2_i$ where $i,j \in \{+ \,,-\}$.

For any bounded function~$V$, let us set $V^{1/2}:=\sgn(V)
|V|^{1/2}$. Define the operator-valued matrix
$$
  B(z ) :=
\begin{pmatrix}
  |V_+|^{1/2}\mathrm{R}_{\alpha ,++}(z) V_+ ^{1/2} \,,
  & |V_+|^{1/2}\mathrm{R}_{\alpha , +-}(z) V_- ^{1/2}
  \\
  |V_-|^{1/2}\mathrm{R}_{\alpha ,-+}(z) V_+ ^{1/2} \,,
  & |V_-|^{1/2}\mathrm{R}_{\alpha , --}(z) V_- ^{1/2}
\end{pmatrix}
$$
acting on $L_+^2\oplus L_-^2$.
%
\begin{theorem} \label{th-resolvent}
Let $z\in \C\setminus \sigma (H_{\alpha , V_+, V_-})$ be such that
the operator $1+B(z)$ is invertible with bounded inverse. Then
\begin{equation}\label{eq-resolvent}
R_{\alpha , V_+,V_-}(z)=R_\alpha (z ) -\sum_{i,j \in \{+,-\}}
\check{R}_{\alpha, i} (z) V_i^{1/2}[1+B(z)]_{ij}^{-1}|V_j|^{1/2}
\hat{R}_{\alpha  ,j} (z )\,
\end{equation}
determines the resolvent of $H_{\alpha ,V_+,V_-}$; the notions
$[1+B(z)]_{ij}^{-1}$ stand for the operator valued matrix elements
of $[1+B(z)]^{-1}$.
\end{theorem}
\begin{proof} Combining (\ref{eq-resolupert-ext}) and
(\ref{eq-aux1}) and using the definition of $R_{\alpha ,
V_+,V_-}(z)$ given by (\ref{eq-resolvent}) we have $\mathrm{Ran}\,
\check{R}_{\alpha , i} (z)\subset W^{1,2}$ and consequently
$\mathrm{Ran}\, R_{\alpha , V_+, V_-} (z)\subset W^{1,2}$.

Denote   $\Upsilon(z)_{ij}:= [1+B(z)]_{ij}^{-1}$ and assume $\psi
\in W^{1,2}$ and $\phi \in L^2$. Employing the identity $
\mathcal{E}_{\alpha, 0, 0}( \psi, R_{\alpha }(z)\phi ) -z ( \psi,
R_{\alpha }(z) \phi ) =(\psi, \phi)$ we get
\begin{eqnarray} \nonumber
\lefteqn{\mathcal{E}_{\alpha, V_+, V_-} \big( \psi, R_{\alpha ,
V_+,V_-}(z)\phi \big) -z \big( \psi, R_{\alpha , V_+,V_-}(z) \phi
\big)
 = (\psi , \phi )}
\\
\nonumber && + \sum_{i \in \{+, -\}} \big(\psi , V_i
\hat{R}_{\alpha , i }(z)\phi\big)_{L^2 _i}- \sum_{i, j\in \{+,
-\}} \big(\psi,
 V_i^{1/2}\Upsilon (z)_{ij}|V_j|^{1/2}
\hat{R}_{\alpha  ,j} (z ) \phi \big)_{L^2 _i}
\\
\nonumber &&- \sum_{i ,j , k\in \{+, -\}} \big(\psi , V_i
\mathrm{R}_{\alpha, ij }(z) V_j ^{1/2} \Upsilon (z)_{jk}
|V_k|^{1/2} \hat{R}_{\alpha  ,k} (z ) \phi \big)_{L^2 _i}\,.
\end{eqnarray}
Furthermore, implementing
$$
  \big(\psi , V_i \hat{R}_{\alpha , i
}(z)\phi\big)_{L^2 _i}= \sum_{j , k\in \{+, -\}} \big(\psi , V_i
^{1/2} \Upsilon(z)_{ik}^{-1}\Upsilon(z)_{kj} |V_j |^{1/2}
\hat{R}_{\alpha , j }(z)\phi\big)_{L^2 _i}
$$
and using the explicit form of $\Upsilon (z)_{ij} $, one obtains
\begin{equation*}
\mathcal{E}_{\alpha, V_+, V_-} \big( \psi, R_{\alpha ,
V_+,V_-}(z)\phi \big) -z \big( \psi, R_{\alpha , V_+,V_-}(z) \phi
\big) =(\psi, \phi)\,.
\end{equation*}
Repeating the same argument as in the proof
Theorem~\ref{th-resolvent0}, we conclude that $R_{\alpha, V_+,
V_-}(z)$ is invertible.
%
\end{proof}

The resolvent formula derived in the above theorem can be written
in a short, more familiar form. For this aim it is convenient to
write $\mathrm{V}^{1/2}=V^{1/2} _+ \oplus V^{1/2} _-$ (analogously
for $|\mathrm{V}|^{1/2}$) and introduce $\hat{R}_\alpha (z)\,:\,
L^2 \to L^2_+ \oplus L^2_- $ defined by $\hat{R}_\alpha (z)\psi
=\hat{R}_{\alpha , +}(z) \psi \oplus \hat{R}_{\alpha , -}(z)\psi $
and $\check{R}_\alpha (z)\,:\, L^2_+ \oplus L^2_- \to L^2 $
defined by $\check{R}_\alpha (z)(f_+ \oplus f_- )
=\check{R}_{\alpha , +}(z) f_+  + \check{R}_{\alpha , -}(z)f_-$.
With these notations, identity~(\ref{eq-resolvent}) reads
\begin{equation}\label{eq-resolI}
R_{\alpha , V_+,V_-}(z)=R_\alpha (z ) - \check{R}_{\alpha} (z)
\mathrm{V}^{1/2}[1+B(z)]^{-1}|\mathrm{V}|^{1/2} \hat{R}_{\alpha }
(z )\,.
\end{equation}

Note that (\ref{eq-resolI}) states a Krein-like  resolvent
studied by Posilicano, \cite{Po, Po04}. (To be fully specific we
identify the embedding $\eta $ introduced in \cite{Po} with
 $W^{1,2} \hookrightarrow  L^2 (\R\times \{+a \}, |V_+|^{1/2}dx) \oplus L^2
(\R\times \{-a\}, |V_-|^{1/2}dx) $. Furthermore, redefining all
the embeddings introduced above by means of $\eta$, we get the
resolvent formula derived in \cite[Thm.~2.1]{Po}.
Applying the results of \cite{Po04},
$$
z\in \rho (H_\alpha )\cap \rho (H_{\alpha, V_+,
V_-})\Leftrightarrow z\in \rho (1+B(z))
$$
and
\begin{equation}\label{eq-spectresol}
z\in \sigma _{\mathrm{p}}(H_{\alpha, V_+, V_-}) \Leftrightarrow
\mathrm{Ker}\, (1+B(z))\neq \{0\}\,.
\end{equation}

\subsubsection{The system with mirror symmetry}\label{Sec.mirror}
Let us consider a system with the mirror symmetry, \ie
\begin{equation}\label{mirror}
V_+ = V_0 = V_- \,,
\end{equation}
where $V_0:\R\to\R$ is a given function. For this case we use a
special notation, $H_{\alpha , V_0}:= H_{\alpha , V_0, V_0}$.
Using~(\ref{eq-resolvent}), we can reconstruct the resolvent
$R_{\alpha ,V_0}(z):=(H_{\alpha , V_0}-z)^{-1}$.

Now, we introduce a ``slight'' perturbation of the symmetry taking
\begin{eqnarray} \label{eq-broken}
&&V_+ =V+V_\epsilon\,, \qquad \mbox{where} \qquad V_\epsilon
:=\epsilon V_p \,, \quad \epsilon> 0 \,,
\\ \nonumber  &&
V_- =V \,,
\end{eqnarray}
where the ``perturbant'' $V_p$ is a function from $L^\infty (\R)$.
In the following we abbreviate $H_{\alpha ;
\epsilon}:=H_{\alpha,V+ V_\epsilon ,V }$.

Considering the system governed by $H_{\alpha , V_0}$ as a
starting point, we construct the resolvent $R_{\alpha ; \epsilon}$
of $H_{\alpha ; \epsilon}$
\begin{equation}\label{eq-resolventperturbed}
R_{\alpha ; \epsilon }(z)=R_{\alpha , V_0}(z ) -
\check{R}_{\alpha, +} (z) V_\epsilon ^{1/2} \Gamma (z) ^{-1}
|V_\epsilon |^{1/2} \hat{R}_{\alpha ,+} (z )\,
\end{equation}
where
\begin{equation}\label{eq-gamma}
  \Gamma(z):=1+|V_\epsilon
|^{1/2}\mathrm{R}_{\alpha , V_0 }(z) V_\epsilon ^{1/2}\,
\end{equation}
and  $\mathrm{R}_{\alpha , V_0 }$ acts as $R_{\alpha , V_0 ,V_0 }
$ but maps from $L_+^2$ to $L_+^2$. The proof of
(\ref{eq-resolventperturbed}) can be obtained repeating the
arguments of the proof of Theorem~\ref{th-resolvent}.
%
%

\section{The essential spectrum}\label{sec-essential}

The spectrum of the unperturbed Hamiltonian $H_{\alpha, 0,0 }$ is
given by (\ref{eq-decomunpert0}). In the following we show that
the essential spectrum~\eqref{spectrum0} is stable under
perturbations~$V_\pm$ which \emph{vanish at infinity} in the following
sense:
\begin{equation}\label{vanish}
  \lim_{L\to\infty}
  \esssup_{\R\setminus[-L,L]} |V_\pm| = 0
  \,.
\end{equation}
\begin{theorem}[Essential spectrum]\label{Thm.ess}
Assume~\eqref{vanish}. Then
$$
  \sigma_\mathrm{ess}(H_{\alpha,V_+,V_-})
  = [\xi_0,\infty)
  \,.
$$
\end{theorem}

We prove the theorem as a consequence of two steps.

\subsection{A lower bound for  the essential spectrum threshold}

We show that the threshold of the essential spectrum does not
descend below the energy~$\xi_0$.
\begin{lemma}
Assume~\eqref{vanish}. Then
$$
  \inf\sigma_\mathrm{ess}(H_{\alpha,V_+,V_-})
  \geq \xi_0
  \,.
$$
\end{lemma}
\begin{proof}
Given a number $L>a$, let~$H^N$ denote the operator
$H_{\alpha,V_+,V_-}$ with a supplementary Neumann condition
imposed on the lines $\{\pm L\} \times \R$ and segments $(-L,L)
\times \{\pm L\}$. It is conventionally introduced as the operator
associated with the quadratic form~$\mathcal{E}^N$ which acts in
the same way as $\mathcal{E}_{\alpha,V_+,V_-}$ but has a larger
domain $
  \Dom(\mathcal{E}^N) := \bigoplus_{k=0}^4 W^{1,2}(\Omega_k)
$, where~$\Omega_k$ are the connected components of~$\R^2$ divided
by the curves where the Neumann condition is imposed:
\begin{equation*}
  \Omega_0 := (-L,L)^2
  \,, \quad
\begin{aligned}
  \Omega_1 &:= (-L,L) \times (-\infty,-L)
  \,, \quad
  & \Omega_3 &:= (-\infty,-L) \times \R
  \,, \\
  \Omega_2 &:= (-L,L) \times (L,\infty)
  \,, \quad
  &\Omega_4 &:= (L,\infty) \times \R
  \,.
\end{aligned}
\end{equation*}
We have the decomposition $H^N = \bigoplus_{k=0}^4 H_k^N$,
where~$H_k^N$ are the operators associated on $L^2(\Omega_k)$ with
the quadratic forms
\begin{align*}
  \mathcal{E}_k^N[\psi]
  & := \int_{\Omega_k} |\nabla\psi|^2
  + \int_{(\R\times\{+a\})\cap\Omega_k} (V_+ - \alpha) |I_+ \psi|^2
  + \int_{(\R\times\{-a\})\cap\Omega_k} (V_- - \alpha) |I_- \psi|^2
  ,
  \\
  \Dom(\mathcal{E}_k^N)
  & := W^{1,2}(\Omega_k )
  \,.
\end{align*}

Since $H_{\alpha,V_+,V_-} \geq H^N$ and the spectrum of~$H_0^N$ is
purely discrete, the minimax principle gives the estimate
$$
  \inf\sigma_\mathrm{ess}(H_{\alpha,V_+,V_-})
  \geq \min_{k\in\{1,\dots,4\}}
  \left\{ \inf\sigma_\mathrm{ess}(H_k^N) \right\}
  \geq \min_{k\in\{1,\dots,4\}}
  \left\{ \inf\sigma(H_k^N) \right\}
  \,.
$$
It is easy to see that the spectra of~$H_1^N$ and~$H_2^N$ coincide
with the non-negative semi-axis $[0,\infty)$. Hence, it remains to
analyze the bottoms of the spectra of~$H_3^N$ and~$H_4^N$. Using
the analogous arguments as leading to Lemma~\ref{Lem.lower}, we
have the lower bound ($k\in\{3,4\}$)
$$
  \mathcal{E}_k^N[\psi] - \xi_0 \, \|\psi\|_{L^2(\Omega_k)}^2
  \,\geq\, \int_{\Omega_k} |\partial_1\psi|^2
  + \int_{\Omega_k} \lambda \, |\psi|^2
  \,\geq\, \essinf_{\Omega_k}\lambda \ \|\psi\|_{L^2(\Omega_k)}^2
  \,.
$$
Consequently,
$$
  \inf\sigma_\mathrm{ess}(H_{\alpha,V_+,V_-})
  \geq
  \min\Big\{ 0, \,
  \xi_0 + \min
  \big\{ \essinf_{\Omega_3}\lambda, \essinf_{\Omega_4}\lambda \big\}
  \Big\}
  \,.
$$
Now, since~$V_\pm$ vanish at infinity, the same is true for the
function~$\lambda$. Hence, for every~$\epsilon$, there exists~$L$
such that for a.e.\ $|x_1|> L$ ($x_2 \in \R$), we have
$
  |\lambda(x)| < \epsilon
$.
Consequently, if $\epsilon < |\xi_0|$,
$$
  \inf\sigma_\mathrm{ess}(H_{\alpha,V_+,V_-})
  \geq \xi_0-\epsilon
  \,.
$$
The claim then follows by the fact that~$\epsilon$ can be chosen
arbitrarily small.
\end{proof}

\subsection{The opposite inclusion}
%
%
\begin{lemma}
Assume~\eqref{vanish}. Then
$$
  \sigma_\mathrm{ess}(H_{\alpha,V_+,V_-})
  \supseteq [\xi_0,\infty)
  \,.
$$
\end{lemma}
\begin{proof}
Our proof is based on the Weyl criterion adapted to quadratic
forms in~\cite{DDI} and applied to quantum waveguides
in~\cite{KKriz}. By this general characterization of essential
spectrum and since the set $[\xi_0,\infty)$ has no isolated
points, it is enough to find for every $\xi \in [\xi_0,\infty)$ a
sequence $
  \{\psi_n\}_{n=1}^\infty \subseteq \Dom(\mathcal{E}_{\alpha,V_+,V_-})
$ such that
\begin{itemize}
\item[(i)]
$\forall n\in\N\setminus\{0\}$, \quad $\|\psi_n\|_{L^2}=1$,
\item[(ii)]
$\big\|(H_{\alpha,V_+,V_-}-\xi)\psi_n\big\|_{-1}
\xrightarrow[n\to\infty]{} 0$.
\end{itemize}
The symbol~$\|\cdot\|_{-1}$ denotes the norm in
$\Dom(\mathcal{E}_{\alpha,V_+,V_-})^*$, the latter being
considered as the dual of the space
$\Dom(\mathcal{E}_{\alpha,V_+,V_-})$ equipped with the norm
$$
  \|\psi\|_{+1} :=
  \sqrt{\mathcal{E}_{\alpha,V_+,V_-}[\psi]+(1+C_0)\|\psi\|_{L^2}^2}
  \,,
$$
where~$C_0$ denotes any positive constant such that
$H_{\alpha,V_+,V_-}+C_0$ is a non-negative operator. We choose the
constant~$C_0$ sufficiently large, so that $\|\cdot\|_{+1}$ is
equivalent to the usual norm in $W^{1,2}$. This is possible in
view of the boundedness of~$V_\pm$ and Lemma~\ref{Lem.Sobolev}.

Let $n\in\N\setminus\{0\}$. Given $k\in\R$, we set
$\xi:=\xi_0+k^2$. Recall that the function $\phi_0$ denotes the
ground state of $-\Delta_\alpha^\R$ corresponding to $\xi_0$.
Since the interactions~$V_\pm$ vanish at infinity, a good
candidates for the sequence $\psi_n$ seem to be plane waves in the
$x_1$-direction ``localized at infinity'' and modulated by
$\phi_0$ in the $x_2$-direction. Precisely
$$
  \psi_n(x) := \varphi_n(x_1) \, \phi_0(x_2) \, e^{ikx_1}
  \,,
$$
where $\varphi_n(x_1):=n^{-1/2} \varphi(x_1/n-n)$ with~$\varphi$
being a non-zero $C^\infty$-smooth function with compact support
in the interval $(-1,1)$. Note that $\supp\varphi_n \subset
(n^2-n,n^2+n)$. Clearly, $\psi_n \in
W^{1,2}=\Dom(\mathcal{E}_{\alpha,V_+,V_-})$. To satisfy~(i), we
assume that both~$\phi_0$ and~$\varphi$ are normalized to~$1$ in
$L^2(\R)$. It remains to verify condition~(ii).

By the definition of dual norm, we have
$$
  \big\|(H_{\alpha,V_+,V_-}-\xi)\psi_n\big\|_{-1}
  = \sup_{\eta \in W^{1,2} \setminus\{0\}}
  \frac{\big|\mathcal{E}_{\alpha,V_+,V_-}(\eta,\psi_n)
  - \xi (\eta,\psi_n)_{L^2}\big|}
  {\|\eta\|_{+1}}
  \,.
$$
An explicit computation using an integration parts yields
\begin{multline*}
  \mathcal{E}_{\alpha,V_+,V_-}(\eta,\psi_n)-\xi(\eta,\psi_n)_{L^2}
  \\
  = \Big(\eta,
  [-\ddot\varphi_n-2ik\dot\varphi_n] \;\! \phi_0 \;\! e^{ikx_1}
  \Big)_{L^2}
  + \int_{\R\times\{a\}} V_+ \,  I_+ (\overline{\eta} \, \psi_n)
  + \int_{\R\times\{-a\}} V_- \, I_- (\overline{\eta} \, \psi_n)
  \,.
\end{multline*}
Using the Schwarz inequality and the normalization of~$\phi_0$, we
get the estimates
\begin{align*}
  \Big|\Big(\eta,
  [-\ddot\varphi_n-2ik\dot\varphi_n] \;\! \phi_0 \;\! e^{ikx_1}
  \Big)_{L^2}\Big|
  &\leq \|\eta\|_{L^2}
  \left(
  \|\ddot\varphi_n\|_{L^2(\R)}+2|k|\|\dot\varphi_n\|_{L^2(\R)}
  \right)
  \,,
  \\
  \Big| \int_{\R\times\{\pm a\}} V_\pm \, I_\pm (\overline{\eta} \, \psi_n )\Big|
  &\leq \|I_\pm \eta\|_{L^2_\pm } \, \|V_\pm\varphi_n\|_{L^2 (\R) } \,
  |\phi_0(\pm a)|
  \,.
\end{align*}
By the choice of~$C_0$ and Lemma~\ref{Lem.Sobolev}, both
$\|\eta\|_{L^2}$ and $\|I_\pm \eta\|_{L^2_\pm}$  can be bounded by
a constant times $\|\eta\|_{+1}$. Hence, there is a constant~$C$,
depending on~$a$, $\alpha$, $k$, $\|V_+\|_\infty$ and
$\|V_-\|_\infty$, such that
\begin{multline*}
  \big\|(H_{\alpha,V_+,V_-}-\xi)\psi_n\big\|_{-1}
  \\
  \leq C \left(
  \|\ddot\varphi_n\|_{L^2(\R)} + \|\dot\varphi_n\|_{L^2(\R)}
  + \|V_+\varphi_n\|_{L^2(\R)} + \|V_-\varphi_n\|_{L^2(\R)}
  \right)
  \,.
\end{multline*}
The first two norms on the right hand side tends to zero as
$n\to\infty$ because
$$
  \|\dot\varphi_n\|_{L^2(\R)}
  = n^{-1} \, \|\dot\varphi\|_{L^2(\R)}
  \,, \qquad
  \|\ddot\varphi_n\|_{L^2(\R)}
  = n^{-2} \, \|\ddot\varphi\|_{L^2(\R)}
  \,.
$$
The remaining terms tend to zero because of the estimate
\begin{equation*}
  \|V_\pm\varphi_n\|_{L^2(\R)}
  \leq
  \esssup_{\supp\varphi_n}|V_\pm|
  \,,
\end{equation*}
in which we have employed the normalization of~$\varphi_n$, and
the fact that $\inf\supp\varphi_n$ tends to infinity as
$n\to\infty$.
\end{proof}
%

\section{The point spectrum} \label{sec-discrete}
%
%
In this section we study the existence of eigenvalues
corresponding to bound states. We will be interested in the
discrete spectrum as well as embedded eigenvalues.

\subsection{The discrete spectrum}
%
In the following we establish a sufficient condition which
guarantees the existence of discrete eigenvalues outside the
essential spectrum.
\begin{theorem}[Discrete spectrum]\label{Thm.disc}
Let $V_- + V_+ \in L^1(\R)$. Assume that
$$
  \int_\R (V_- + V_+)(x) \, \D x < 0
  \,.
$$
Then
$$
  \inf\sigma(H_{\alpha,V_+,V_-}) < \xi_0
  \,.
$$
Consequently, if in addition~$V_\pm$ vanish at
infinity~\eqref{vanish}, then~$H_{\alpha,V_+,V_-}$ possesses at
least one isolated eigenvalue of finite multiplicity
below~$\xi_0$, \ie~\eqref{DiscSpec} holds.
\end{theorem}
\begin{proof}
The proof is based on the variational argument. Our aim is to find
a test function $\psi \in W^{1,2}$ such that
$$
  Q[\psi] :=
  \mathcal{E}_{\alpha,V_+,V_-}[\psi] - \xi_0 \, \|\psi\|_{L^2(\R^2)}^2 < 0
  \,.
$$
For any $n\in\N \setminus \{0\}$, we set
$$
  \psi_n(x) := \varphi_n(x_1) \phi_0(x_2)
  \,,
$$
where~$\phi_0$ is, as before, the positive eigenfunction of
$-\Delta_\alpha^\R$, normalized to~$1$ in $L^2(\R)$, and
\begin{equation}\label{capacity}
  \varphi_n(x_1) :=
  \begin{cases}
    1
    & \mbox{if} \quad |x_1| \leq n \,, \\
    \displaystyle
    \frac{2n - |x_1|}{n}
    & \mbox{if} \quad n \leq |x_1| \leq 2 n \,, \\
    0
    &\mbox{otherwise} \,.
  \end{cases}
\end{equation}
Obviously, $\psi_n \in W^{1,2}$. Using
$\varepsilon_\alpha[\phi_0]=\xi_0 \|\phi_0\|_{L^2(\R)}^2$ and the
fact that~$\phi_0$ is even, it is easy to check the identity
$$
  Q[\psi_n]
  = \int_\R |\varphi_n'|^2
  + |\phi_0(a)|^2 \int_\R (V_- + V_+) \, |\varphi_n|^2
  \,.
$$
As $n\to\infty$, the first term on the right hand side tends to
zero, while the second converges (by dominated convergence
theorem) to a multiple of $\int_\R (V_- + V_+)<0$. Hence,
$Q[\psi_n]$ is negative for~$n$ sufficiently large.
\end{proof}
\begin{remark}
The integrability of~$V_++V_-$ is just a technical assumption in
Theorem~\ref{Thm.disc}. It is clear from the proof that it is only
important to ensure that the quantity
$$
  \int_\R (V_- + V_+) \, |\varphi_n|^2
$$
becomes negative as $n \to \infty$, the value~$-\infty$ for the
limit being admissible in principle. For instance, one can
alternatively assume that~$V_++V_-$ is non-trivial and
non-positive on~$\Omega_0$ for the present proof to work.
\end{remark}
%

\subsection{Embedded eigenvalues}\label{sec-embedded}
%
The aim of this section is to show that the system with mirror
symmetry (\cf~Section~\ref{Sec.mirror}) possesses embedded
eigenvalues under certain assumptions. For $V_0\in L^\infty(\R)$
vanishing at infinity,
we consider the Hamiltonian $H_{\alpha,V_0}
\equiv H_{\alpha, V_0,V_0}$. Recall that the essential spectrum
of~$H_{\alpha,V_0}$ is given by~\eqref{EssSpec}.

We simultaneously consider an auxiliary operator $H^D _{\alpha
,V_0}$ on  $L^2(\R\times \R^+)$, which acts as $H_{\alpha,V_0}$ on
$\R\times \R^+$, subject to Dirichlet boundary conditions on
$\R\times\{0\}$. It is introduced as the operator on $L^2(\R\times
\R^+)$ associated with the form
\begin{align*}
  \mathcal{E}_{\alpha,V_0}^D[\psi]
  & := \int_{\R\times \R^+} |\nabla\psi|^2
  + \int_{\R\times\{a\}} (V_0 - \alpha) \, |I_+ \psi|^2
  \,,
  \\
  \Dom(\mathcal{E}_{\alpha,V_0}^D)
  & := W_0^{1,2}(\R\times \R^+)
  \,,
\end{align*}
where we keep the same notation $I_+$ for the embedding $I_+ \,:\,
W^{1,2}_0(\R\times \R^+) \hookrightarrow L^2_+$. Since~$V_0$
vanishes at infinity, it can be shown in the same way as in
Section~\ref{sec-essential} that
$$
  \sigma _{\mathrm{ess}} (H^D _{\alpha ,V})
  =[\mu_0,\infty )
\,,
$$
where $\mu _0 \in (\xi_0,0]$ is the spectral threshold of the
one-dimensional operator~$h_\alpha^D$ on $L^2(\R_+)$ associated
with the form
\begin{equation*}
  \varepsilon_{\alpha}^D[\phi]
  := \int_{\R_+} |\phi'|^2
  - \alpha \, |\phi(a)|^2
  \,, \qquad
  \Dom(\varepsilon_{\alpha}^D)
  := W_0^{1,2}(\R_+)
  \,.
\end{equation*}
\begin{proposition} \label{prop-Dirichletev}
One has $
  \sigma _{\mathrm{ess}} (h_\alpha^D)
  = [0,\infty )
$
and
$$
  \sigma _{\mathrm{disc}} (h_\alpha^D)  =
  \begin{cases}
    \{\xi_1\}
    &\mbox{if} \quad \alpha a >1
    \,,
    \\
    \emptyset
    &\mbox{if} \quad \alpha  a \leq 1
    \,.
  \end{cases}
$$
Consequently,
\begin{equation}\label{eq-defmu}
\mu_0 =
\begin{cases}
 \xi_1
  &\mbox{if} \quad \alpha a >1
  \,,
  \\
 0
  &\mbox{if} \quad \alpha  a \leq 1
  \,.
\end{cases}
\end{equation}
\end{proposition}
\begin{proof}
By the methods of Section~\ref{sec-essential}, it is easy to see
that the essential spectrum of~$h_\alpha^D$ coincides with the
non-negative semi-axis.
Note that the existence of a negative eigenvalue of~$h_\alpha^D$
is equivalent to the existence of an eigenfunction of~$-\Delta
_\alpha ^\R$ corresponding to a negative eigenvalue and vanishing
on $\R\times\{0\}$. The latter holds if, and
only if,  the operator $
 -\Delta _\alpha ^\R$ possesses the \emph{odd}
eigenfunction $\phi_1$, \ie\ $\alpha a
 >1$. This means that  $h_\alpha ^D$ has one negative eigenvalue
 which is given by $\xi_1$ if, and only if,
 $\alpha a >1 $; otherwise $\sigma _{\mathrm{disc}}(h_\alpha ^D)
 =\emptyset$.
\end{proof}

\begin{theorem}[Embedded eigenvalues]\label{Thm.mirror}
Let~$V_0$ be vanishing at infinity~\eqref{vanish}.
Assume that $H^D _{\alpha
,V_0}$ possesses a (discrete) eigenvalue~$\nu$ below~$\mu_0$.
Then~$\nu$ is an eigenvalue of $H_{\alpha,V_0}$. More
specifically, if $\nu\geq\xi_0$ (respectively, $\nu<\xi_0$),
then~$\nu$ is an embedded (respectively, discrete) eigenvalue of
$H_{\alpha,V_0}$.
\end{theorem}
\begin{proof}
Let~$\psi$ be an eigenfunction of~$H^D _{\alpha ,V_0}$
corresponding to~$\nu$. Due to the mirror symmetry, the odd
extension of~$\psi$ to~$\R^2$ is an eigenfunction
of~$H_{\alpha,V_0}$ corresponding to the same value~$\nu$. The
rest follows from the fact that the essential spectrum
of~$H^D_{\alpha ,V_0}$ is strictly contained in the essential
spectrum of~$H_{\alpha ,V_0}$
\end{proof}

The following result makes Theorem~\ref{Thm.mirror} non-void.
\begin{proposition}\label{Prop.non-void}
Assume $\alpha a > 1$. Let
$V_0 \in L^1(\R)$ be vanishing at infinity and $\int_\R V_0(x) \,
\D x < 0$. Then $\sigma_\mathrm{disc}(H^D _{\alpha ,V_0}) \not=
\emptyset$.
\end{proposition}
\begin{proof}
It follow from Proposition~\ref{prop-Dirichletev} that,
under the condition $\alpha a > 1$,
$h_\alpha^D$~possesses a negative eigenvalue~$\xi_1$
with eigenfunction~$\phi_1$ (restricted to $\R\times\R_+$).
Then the claim follows by using the test function
$\psi(x):=\varphi_n(x_1)\phi_1(x_2)$, where~$\varphi_n$ is
introduced in~\eqref{capacity}, as in the proof of
Theorem~\ref{Thm.disc}.
\end{proof}

Finally, to ensure the existence of embedded eigenvalues by
Theorem~\ref{Thm.mirror}, it remains to verify that the discrete
eigenvalue~$\nu$ of $H^D _{\alpha ,V_0}$ (which exists under the
hypotheses of Proposition~\ref{Prop.non-void}) can be made larger
than or equal to~$\xi_0$. However, this happens, for instance, in
the weak-coupling regime.
\begin{corollary} \label{th-emb}
Let $a \alpha > 1$. Let $V_0 \in L^1(\R)$ be vanishing at infinity
and $\int_\R V_0(x) \, \D x < 0$. Then there exists $\varepsilon
>0$ such that $H_{\alpha ,\varepsilon V_0}$ has at least one
embedded eigenvalue in the interval $[\xi_0,\mu_0)$.
\end{corollary}
%

\section{Hardy inequalities}\label{sec-Hardy}
%
In this section we study the case when~$V_+$ and~$V_-$ are
non-negative. It is easy to see that, in this situation, the
spectrum does not start below~$\xi_0$. The purpose of this
subsection is to show that a stronger result holds in the latter
setting. We derive a functional, Hardy-type inequality for
$H_{\alpha,V_+,V_-}$ with non-trivial non-negative~$V_+$
and~$V_-$, quantifying the repulsive character of the line
interactions in this case.

A Hardy-type inequality follows immediately from
Lemma~\ref{Lem.lower}. Indeed, neglecting the kinetic term
in~\eqref{pre.Hardy}, we arrive at
\begin{theorem}[Local Hardy inequality]\label{Thm.Hardy}
Assume that~$V_\pm$ are non-negative and that~$V_-$ or~$V_+$ is
non-trivial. Then
$$
  H_{\alpha,V_+,V_-} - \xi_0 \ \geq \ \lambda
$$
in the sense of quadratic forms, where~$\lambda$ is a non-trivial
and non-negative function, \cf~\eqref{lambda}.
\end{theorem}

This Hardy inequality is called \emph{local} since it reflects the
local behaviour of the functions~$V_\pm$. In particular, if~$V_\pm$
is compactly supported then ~$\lambda$ is compactly supported as
the function of the first variable.

In any case, a \emph{global} Hardy inequality follows by applying
the classical Hardy inequality.
\begin{theorem}[Global Hardy inequality]\label{Thm.Hardy.global}
Assume that~$V_\pm$ are non-ne\-gative, and that there exists
$x_1^0\in\R$ and positive numbers~$V_0$ and~$R$ such that
$$
  \forall x_1 \in (x_1^0-R,x_1^0+R), \qquad
  V_+(x_1) \geq V_0
  \quad\mbox{or}\quad
  V_-(x_1) \geq V_0
  \,.
$$
Then
$$
  H_{\alpha,V_+,V_-} - \xi_0 \ \geq \ c \;\! \rho
  \qquad\mbox{with}\qquad
  \rho(x) := \frac{1}{1+(x_1-x_1^0)^2}
$$
in the sense of quadratic forms. Here the constant~$c$ depends
on~$V_0$, $R$ and~$a$.
\end{theorem}
\begin{proof}
The proof follows the strategy developed in \cite[Sec.~3.3]{EKK}
for establishing a similar global Hardy inequality in curved
waveguides. For clarity of the exposition, we divide  the proof
into several steps. Denote $I:=(x_1^0-R,x_1^0+R)$.

\smallskip\noindent
\emph{Step 1.} The main ingredient in the proof is the classical
one-dimensional Hardy inequality
\begin{equation}\label{HI.classical}
  \forall \phi \in W_0^{1,2}(\R_+)
  \,, \qquad
  \int_{\R_+} |\phi'(x)|^2 \, \D x
  \geq \frac{1}{4} \int_{\R_+}
  \frac{|\phi(x)|^2}{x^2} \, \D x
  \,.
\end{equation}
%
We apply it in our case as follows. Let us
define an auxiliary function $\eta:\R\to\R$ by
$\eta(x_1):=|x_1-x_1^0|/R$ if $|x_1-x_1^0|<R$ and by setting it
equal to~$1$ otherwise; we keep the same notation~$\eta$ for the
function $x\mapsto\eta(x_1)$ on $\R^2$. For any $\psi \in
C_0^\infty(\R^2)$, let us write $\psi=\eta\psi+(1-\eta)\psi$.
Applying the classical Hardy inequality to the function $\eta\psi$
and using Fubini's theorem we get
\begin{align}\label{bound0}
  \int_{\R^2} \rho |\psi|^2
  &\leq  2\int_{\R^2} \frac{|\eta\psi|^2}{(x_1-x_1^0 )^2}
  +  2 \int_{\R^2} |(1-\eta)\psi|^2
  \nonumber \\
  &\leq 16 \int_{\R^2} |\partial_1\eta|^2 |\psi|^2
  + 16 \int_{\R^2} |\eta|^2 |\partial_1\psi|^2
  +  2 \int_{I \times \R} |\psi|^2
  \nonumber \\
  &\leq 16 \int_{\R^2} |\partial_1\psi|^2
  + (2+16/R^2) \int_{ I \times  \R } |\psi|^2
  \,.
\end{align}
By the density, the inequality extends to all $\psi \in
W^{1,2}=\Dom(\mathcal{E}_{\alpha,V_+,V_-})$.

\smallskip\noindent
\emph{Step 2.}~By Theorem~\ref{Thm.Hardy}, we have
\begin{equation}\label{bound1}
  \mathcal{E}_{\alpha,V_+,V_-}[\psi] - \xi_0 \, \|\psi\|_{L^2}^2
  \geq \essinf_{I \times \R} \lambda \int_{I \times \R} |\psi|^2
  \geq \lambda_0 \int_{I \times \R} |\psi|^2
\end{equation}
for every $\psi \in \Dom(\mathcal{E}_{\alpha,V_+,V_-})$, where $
  \lambda_0 := \tilde{\lambda}(V_0,0) = \tilde{\lambda}(0,V_0)
$. Of course, $\lambda_0$ is a positive number under the stated
hypotheses.

\smallskip\noindent
On the other hand, neglecting the non-negative potential term
in~\eqref{pre.Hardy}, we have
\begin{equation}\label{bound2}
  \mathcal{E}_{\alpha,V_+,V_-}[\psi] - \xi_0 \, \|\psi\|_{L^2}^2
  \geq \int_{\R^2} |\partial_1\psi|^2
\end{equation}
for every $\psi \in \Dom(\mathcal{E}_{\alpha,V_+,V_-})$.

\smallskip\noindent
\emph{Step 3.}~Interpolating between the bounds~\eqref{bound1}
and~\eqref{bound2}, and using~\eqref{bound0} in the latter, we
arrive at
\begin{equation*}
  \mathcal{E}_{\alpha,V_+,V_-}[\psi] - \xi_0 \, \|\psi\|_{L^2}^2
  \geq \frac{\epsilon}{16}
  \int_{\R^2} \rho |\psi|^2
  \\
  + \left[
  (1-\epsilon) \lambda_0 - \epsilon \left(\frac{1}{8}+\frac{1}{R^2}\right)
  \right]
  \int_{I \times \R} |\psi|^2
\end{equation*}
for every $\psi \in \Dom(\mathcal{E}_{\alpha,V_+,V_-})$ and
$\epsilon\in(0,1)$. It is clear that the right hand side of this
inequality can be made non-negative by choosing~$\epsilon$
sufficiently small. Choosing~$\epsilon$ such that the expression
in the square brackets vanishes, the inequality passes to the
claim of the theorem with
$$
  c := \frac{\lambda_0/16}{\lambda_0 + 1/8 + 1/R^2}
  \,.
$$

It remains to realize that~$\lambda_0$ depends on~$V_0$ and~$a$
through the definition~\eqref{lambda.tilde}.
\end{proof}

As a direct consequence of Theorem~\ref{Thm.Hardy.global}, we get
\begin{corollary}\label{Corol.subcritical}
Assume the hypotheses of Theorem~\ref{Thm.Hardy.global}. Let~$W$
be the multiplication operator in $L^2$ by any bounded
function~$w$ for which there exists a positive constant~$C$ such
that $ |w(x)| \leq C |x_1|^{-2}$ for a.e.\ $x\in\R^2$. Then there
exists $\epsilon_0>0$ such that for every $\epsilon < \epsilon_0$,
$$
  \inf\sigma(H_{\alpha,V_+,V_-} + \epsilon W) \geq \xi_0
  \,.
$$
\end{corollary}

Assume that~$V_\pm$ vanish at infinity. Since also the
potential~$W$ of the corollary is bounded and vanishes at
infinity, it is easy to see that the essential spectrum is not
changed, \ie, $
  \sigma_\mathrm{ess}(H_{\alpha,V_+,V_-} + \epsilon W)
  = [\xi_0,\infty)
$, independently of the value of~$\epsilon$ and irrespectively of
the signs of~$V_\pm$. It follows from the corollary that a certain
critical value of~$\epsilon$ is needed in order to generate
discrete spectrum of $H_{\alpha,V_+,V_-} + \epsilon W$ if the
Hardy inequality for~$H_{\alpha,V_+,V_-}$ exists. On the other
hand, it is easy to see that $H_{\alpha,0,0} + \epsilon W$
possesses eigenvalues below~$\xi_0$ for arbitrarily
small~$\epsilon$ provided that~$W$ is non-trivial and
non-positive.

\section{Resonances induced by broken symmetry}\label{sec-resonances}
%
As was already stated (\cf~Corollary~\ref{th-emb}), the
Hamiltonian $H_{\alpha, V_0}$ of the system with mirror
symmetry~\eqref{mirror} admits embedded eigenvalues. In the
following we show that breaking this symmetry by~(\ref{eq-broken})
will turn the eigenvalues into resonances.

The strategy we employ here is as follows. Our first aim is to
show that the operator-valued function $z\mapsto R _{\alpha ;
\epsilon} (z )$ has a second sheet analytic continuation
in the following sense: for any $f\,,g\in C_0 (\R^2 ) $ the
operator $f R_{\alpha ; \epsilon}(z)g$ can be analytically
continued to the lower half-plane as a bounded operator in $L^2$.
Such a continuation we will denote as $R^\mathit{II} _{\alpha ;
\epsilon} (\cdot )\equiv R^\mathit{II} _{\alpha ; \epsilon , f,g}
(\cdot )$. Of course, the above formulation implies that the
function $z\mapsto (f,R_{\alpha ; \epsilon} (z)g)$ has the second
sheet continuation. The analogous definition will be employed for
the second sheet continuation of $R_{\alpha }(\cdot )$. To recover
resonances in the system governed by~$H_{\alpha ; \epsilon}$, we
look for poles of $R^\mathit{II} _{\alpha ; \epsilon} (\cdot )$.
These poles are defined by the condition
\begin{equation}\label{eq-kernel}
\mathrm{Ker}\, \Gamma ^\mathit{II} (z) \neq \{0\}\,, \qquad \Im
z<0 \,,
\end{equation}
where $\Gamma ^\mathit{II} (\cdot)$ is the second sheet
continuation of analytic operator valued-function $z\mapsto \Gamma
(z)$. Our aim is to find $z$ satisfying (\ref{eq-kernel}).

\subsection{The second sheet continuation of $\Gamma (\cdot)$}
%
Henceforth we assume
\begin{equation}\label{eq-assumptionV}
  V_{\pm}\, \e ^{C|x|}\in L^\infty (\R) \quad \mathrm{for}\quad
  \mathrm{some} \quad C>0
  \,.
\end{equation}

The first auxiliary statement is contained in the following lemma.
\begin{lemma} \label{le-compactness}
Suppose~\eqref{eq-assumptionV}. Then for any $i\,,j\in \{+, -\}$
and $z\in \C \setminus [\xi_0 , \infty) $ the operator $|V_i
|^{1/2} \mathrm{R}_{\alpha , i j} (z) V_j ^{1/2}$ is
Hilbert-Schmidt. Consequently,  $B(z)$ is a
Hilbert-Schmidt operator as well.
\end{lemma}
\begin{proof}
 Using (\ref{eq-resolupert-ext}) we have
\begin{multline*}
|V_i|^{1/2} \mathrm{R}_{\alpha, ij}(z) V_j^{1/2}
\\
  =
\underbrace{|V_i|^{1/2} \mathrm{R}_{0 ,ij} (z) V_j^{1/2}}_{A_1} -
\sum_{k,l\in \{+,-\}}
  \underbrace{|V_i|^{1/2}\mathrm{R}_{0,i k} (z)}_{A_2}
\Gamma _{0,\alpha }(z)_{kl}^{-1} \underbrace{\mathrm{R}_{0 ,lj} (z
)V_j^{1/2}}_{A_3}
  \,.
\end{multline*}
It is well known that $R_0 (z)$ is an integral operator
$$
R_0 (z)f (x)= \int_{\R^2}G_0 (z; x -w)f(w)\mathrm{d}w\,,\quad G_0
(z ; x )=\frac{1}{2\pi}K_0 (\sqrt{z}|x|)\,,
$$
where $K_0 (\cdot)$ stands for the Macdonald function
(\cf~\cite[Sec.9.6]{AS}), $\Im \sqrt{z} >0$. Consequently,
$\mathrm{R}_{0, ij }(z)$ is an integral operator with the kernel
$G_{0,ij}(z; \cdot-\cdot )$ defined as the ``bilateral'' embedding
of  $G_0 (z; \cdot -\cdot )$ acting from $L^2_j $ to $L^2_i$.
Using the properties of $K_0$ (see \cite[Eq.~(9.6.8)]{AS}), we
conclude that $G_0 (z;\cdot )$ has a logarithmic singularity at
the origin and apart from~$0$ it is continuous; moreover it
exponentially decays at the infinity. This implies that
$G_{0,ij}(z; \cdot)\in L^2 (\R)$ and consequently $|G_{0,ij}(z;
\cdot)|^2 \in L^1 (\R)$. Since $V_\pm \in L^\infty (\R)\cap L^1
(\R)$ we have
\begin{align*}
\| |V_i |^{1/2} \mathrm{R} _{0, ij} (z) V_j ^{1/2}\|^2
_\mathrm{HS} & \leq \int_{\R^2} |V_i (x)| |G_{0,ij}(z; x-y)|^2
|V_j (y) | \, \mathrm{d}x
 \mathrm{d}y\,, \\
 &
 \leq
\|V_j\|_{\infty}\| |V_i|^{1/2}\mathrm{R}_{0,
ij}(z)\|_\mathrm{HS}^2 \\ & \leq \|V_j\|_{\infty} \|V_{i}\|_{L^1
(\R)} \||G_{0,ij}(z)|^2\|_{L^1 (\R)}\,,
\end{align*}
where $\|\cdot\|_\mathrm{HS}$ denotes the Hilbert-Schmidt norm and
the last step employs the Young inequality (\cf~\cite[Chap.~IX.4,
Ex.~1]{RS2}). The above inequalities show that the operators
called $A_1\,, A_2\,,A_3$  are the Hilbert-Schmidt operators; note
$\| \mathrm{R}_{0, ij}(z)V_j^{1/2}\|_\mathrm{HS}
=\||V_j|^{1/2}\mathrm{R}_{0, ji}(z)\|_\mathrm{HS}$. Moreover,
since $\Gamma _{0,\alpha }(z)_{kl}^{-1}$ are bounded, $|V_i|^{1/2}
\mathrm{R}_{\alpha, ij}(z) V_j^{1/2}$ are Hilbert-Schmidt as well.
\end{proof}
\begin{remark} \rm{Note that to prove the above lemma we  use only $V_\pm \in
L^\infty (\R)\cap L^1 (\R)$; the stronger assumption
(\ref{eq-assumptionV}) will be used in the following.}
\end{remark}

Suppose $\mathcal{B}$ is an open set from  $[\xi_0 , \infty )$ and
$E_{\mathcal{B}}$ denotes the spectral measure of $H_{\alpha , V_+
, V_-}$.
Denote $\mathcal{H}_{ac}= \{\psi \in L^2\,:\, \mathcal{B}\mapsto
(\psi, E_{\mathcal{B}} \psi)\,\,
\mathrm{is}\,\,\mathrm{absolutely} \,\,\mathrm{continuous}\}$.
\begin{lemma} \label{le-acspectrum}
Suppose~\eqref{eq-assumptionV}.
\begin{enumerate}
\item
For any interval $[a,b ]\subset [\xi_0 , \mu_0 ] $ which is
disjoint from a finite set
of numbers~$\mathcal{E}$ we have  $ \mathrm{Ran }\,
E_{(a,b)}\subset \mathcal{H}_{ac}$.

\item There exists a region $\Omega_- \subset \C_-$ with a
boundary containing $(\xi_0\,,\mu_0) $ and operator-valued
function $R^{\mathit{II}}_{\alpha, V_+, V_-}(\cdot )$  analytic in
$ \Omega\setminus \mathcal{E}$, where $ \Omega = \C_+ \cup (\xi_0
, \mu_0)\cup \Omega_- $ which constitutes the analytic
continuation of $R_{\alpha, V_+, V_-}(\cdot )$.
\end{enumerate}
\end{lemma}
\begin{proof}
%
Operator $R_{\alpha}(\cdot )$ is analytic in $\C\setminus [\xi_0,
\infty)$. The Stone formula implies that the limit
$\mathrm{s}\mbox{--}\!\lim_{\varepsilon\to 0}\left( R_\alpha
(\lambda +i\varepsilon)- R_\alpha (\lambda -i\varepsilon) \right)
\neq 0$ for $\lambda \in (\xi_0, \mu_0)$. In the following we will
use notations $R_\alpha (\lambda \pm i0)$ for the limits
$\mathrm{s}\mbox{--}\!\lim_{\varepsilon \to 0}R_\alpha (\lambda
\pm i\varepsilon)$ and analogously for the resolvents of the
remaining operators. Our first aim is to show that $R_\alpha
(\cdot )$ can be analytically continued from $\C_+$ through
$(\xi_0, \mu_0 )$ to the lower half-plane in the sense described
at the beginning of this section. Recall
\begin{equation}\label{eq-resoldecomposition}
  R_\alpha (z)= R^c _\alpha (z)+R^d (z)\,,\quad R^d (z)=
  \sum _{j\in \mathcal{N}}R_\alpha ^{j,d} (z)\,.
\end{equation}
Note, that $R^c _\alpha (z)$ is analytic for $z \in \C\setminus
[0, \infty )$, \cf~(\ref{eq-Rc}). Furthermore, if $\alpha a >1$
then there exists the component $R^{1,d }_\alpha (z)$ of $R^{d
}_\alpha (z)$ which is analytic for $z\in \C\setminus [\xi_1 ,
\infty)$. On the other hand, the analytic continuation of $R^{0,d}
 _\alpha (\cdot )$ through $(\xi_0, \infty)$ is defined by
the second sheet values of $\tau_0 $, \ie\ $\Im \tau_0
^\mathit{II}\leq 0$. Precisely for $f\,,g\in C_0 (\R^2)$ the
operator  $fR^{0,d}
 _\alpha (\cdot )g$ has a second sheet analytical continuation
which is an integral operator with the kernel $f(x)G^{0,d,
\mathit{II}}
 _\alpha (z; x-y )g(y)$ and $G^{0,d,
\mathit{II}}
 _\alpha$ is defined by (\ref{eq-Rd}) after substituting
 $\tau_0 $ by  $ \tau_0
^\mathit{II}$. The resulting operator we denote as $R^{0,d,
\mathit{II}}
 _\alpha (\cdot )$.
  Consequently, we define the
second sheet continuation of $R
 _\alpha (z)$  as
$R^{\mathit{II}}_\alpha (z)= R_\alpha ^c (z)+R^{ d, \mathit{II}}
 _\alpha (z)$ where $R^{d,
\mathit{II}}
 _\alpha (z)=R^{0, d, \mathit{II}}
 _\alpha (z)+R^{1, d}
 _\alpha (z)$ and $z\in \C_+ \cup (\xi_0 , \mu_0  )
 \cup \C_- $. Operator $R^{\mathit{II}}_\alpha (z)$ is analytic bounded in $L^2$
 in the sense described at the beginning of section.
\\
By means of $G^{0,d, \mathit{II}}
 _\alpha (z; x-y )$ we define the operator $|V_i|^{1/2} \mathrm{R}_{ \alpha ,ij}^{0,d, \mathit{II}}
(z)V_j^{1/2}\,:\, L_j^2 \to L_i^2$, $i\,,j \in \{+\,,-\}$ with the
HS-norm
\begin{multline*}
\big\||V_i|^{1/2} \mathrm{R}_{ \alpha ,ij}^{0,d, \mathit{II}}
(z)V_j^{1/2}\big\|^2_\mathrm{HS}
\\
=\frac{|\phi_0 (i a)\phi_0 (j a)|^2}{4 |\tau _0^{\mathit{II}}
(z)|^2 }\int_{\R^2} \e^{-C(|x_1|+|y_1|)} \e ^{-2\Im
\tau_0^{\mathit{II}} (z)|x_1-y_1|} \,|h_i(x_1)
h_j(y_1)|\mathrm{d}x_1 \mathrm{d}y_1\,,
\end{multline*}
where $h_\pm := V_\pm \e^{C|\cdot|}\in L^\infty (\R)$,
\cf~(\ref{eq-assumptionV}). The above expression is finite if $\Im
\tau_0^{\mathit{II}} (z)>- C/2$. For $|V_i|^{1/2} \mathrm{R}_{
\alpha ,ij}^{1,d} (z)V_j^{1/2}$ the analogous expression is always
finite because $R_\alpha ^{1,d} (z)$ is analytic on $\C_+\cup
(\xi_0\,,\mu_0 ) \cup \C_- $ and $\Im \tau_1 (z)
>0$.  Furthermore, since $|V_i|^{1/2} \mathrm{R}_{\alpha,ij}
(z)V_j^{1/2}$, $z\in\C\setminus [\xi_0, \infty)$ is a
Hilbert-Schmidt operator (\cf~Lemma~\ref{le-compactness}) as well
as  $|V_i|^{1/2} \mathrm{R}^{d}_{\alpha, ij} (z)V_j^{1/2}$, we
conclude that $|V_i|^{1/2} \mathrm{R}^{c}_{\alpha,ij}
(z)V_j^{1/2}$ is also the Hilbert-Schmidt operator with the
 boundary values $|V_i|^{1/2} \mathrm{R}^{c}_{\alpha, ij}
(\lambda \pm i 0 )V_j^{1/2}$, $\lambda \in (\xi_0 \,,\mu_0)$ being
compact.
Con\-se\-quent\-ly, $B_{ij} (\lambda +i0)$, $\lambda \in (\xi_0,
\mu_0)$ is compact and $B_{ij}(\cdot)$ has the second sheet
continuation $B_{ij}^{\mathit{II}}(\cdot )=|V_i|^{1/2}
\mathrm{R}_{\alpha , ij}^{\mathit{II}} (\cdot)V_j^{1/2}$ through
$(\xi_0, \mu_0)$. Finally, we conclude that $B^{\mathit{II}}(z)$
is compact $z\in\Omega =\C_+\cup (\xi_0, \mu_0)\cup \Omega_-$,
where $\Omega_-$ is a region in $\C_-$ with boundary containing
$(\xi_0, \mu_0)$ and confined by the condition $\Im \tau_0 (z)>-
C/2$.
\\
Note that the  operator $f \check{R}_{\alpha ,i }(z)V_i ^{1/2}$,
where $f\in C_0 (\R^2)$, has the analytic second sheet
continuation. Indeed, as the above discussion shows the only
nontrivial component is given by $f \check{R}^{0,d}_{\alpha ,i
}(z)V_i ^{1/2}$. Since its Hilbert-Schmidt norm takes the form
\begin{multline*}
\|f \check{R}^{0,d, \mathit{II} }_{\alpha ,i }(z) V_i
^{1/2}\|^2_\mathrm{HS}=
\\
=\frac{|\phi_0 (ia)|^2 }{4| \tau_0^{\mathit{II}} (z)|^2}
\int_{\R^2} \int_{\R} \mathrm{e}^{-C|y_1|} \mathrm{e} ^{-2 \Im
\tau_0 ^{\mathit{II} }(z)|x_1-y_1 |}
 | f(x)|^2 |h_i (y_1 )|
\mathrm{d}x \mathrm{d}y_1\,,
\end{multline*}
where $x=(x_1, x_2)$ and $h_i:= V_i \e^{C|\cdot|}\in L^\infty
(\R)$ we can conclude that $f \check{R}^{0,d, \mathit{II}}_{\alpha
,i }(z) V_i ^{1/2}$ is analytic for $z\in \Omega $ and bounded as
the operator acting from $L^2_i$ to $L^2$ . The analogous
statement can be obtained for $|V_i|^{1/2}\hat{R}_{\alpha , i
}(z)f$.
\medskip \\
\emph{Ad~1.} Since the condition $\mathrm{Ker}\,[1+B(z)] \neq
\{0\}$ determines poles of $R_{\alpha, V_+, V_-}(\cdot )$,
\cf~(\ref{eq-spectresol}), which is the resolvent of a
self-adjoint operator, the former has no solution for $z \in
\C_+$. Combining compactness of $B(z)$ and the analytic Fredholm
theorem (see, \eg, \cite[Thm. VI.14]{RS1}) together with the fact
that $z\mapsto B(z)$ is analytic for $z\in \C_+ \cup (\xi_0 ,
\mu_0)$, we conclude that the operator $[1+B(z)]^{-1}$ exists and
it is bounded analytic in $z \in \C_+$ with the boundary values
$z=\lambda +i 0$, provided $\lambda \in (\xi_0 , \mu_0)$ avoids a
finite set $\mathcal{E}$ of real numbers; for an analogous
discussion see~\cite[Thm.~XIII.21]{RS4}.
Moreover, the operator $|f|
 R_{\alpha 
}(\lambda +i \varepsilon)|f|$, $f\in C_0 (\R^2 )$ is uniformly
bounded for $0< \varepsilon < 1$ and $\lambda \in (\xi_0, \mu_0)$.
The operators $|f| \check{R}_{\alpha ,i }(z)V_i ^{1/2}$ and
$|V_i|^{1/2}\hat{R}_{\alpha , i }(z)|f|$ are uniformly bounded as
well. Combining the above statements with the resolvent
formula~(\ref{eq-resolvent}), we come to the conclusion that, for
any $f\in C_0 (\R^2)$, the function
\begin{multline*}
 |f|
 R_{\alpha , V_+, V_-
}(z)|f|
 =  |f| R_{\alpha}(z)|f|
\\
- \sum_{i , j \in \{+, -\}}(|f| \check{R}_{\alpha ,i }(z)V_i
^{1/2})
  \left[ 1+B (z)\right]_{ij}^{-1}
(|V_j |^{1/2} \hat{R}_{\alpha , j }(z)|f|  )\,,
\end{multline*}
is uniformly bounded with respect to $z=\lambda  +i \varepsilon $,
$0<\varepsilon <1$, $\lambda \in [a,b]$. Furthermore, since
$$
|(g,R_{\alpha , V_+, V_- }(z) g)| \leq \||g|^{1/2}\|^2\,
\|\,|g|^{1/2}R_{\alpha , V_+, V_- }(z)|g|^{1/2}\| \,,
$$
for any $g\in C_0 (\R^2)$ the assumption of
\cite[Thm.~XIII.19]{RS4} is fulfilled and it yields the claim.
(Note that in the last expression the same notion $\|\cdot \|$ was
used as the operator norm as well as the norm of function).
\medskip \\
\emph{Ad~2.} Using again the Fredholm theorem and the compactness
of $B^{\mathit{II}}(\cdot)$, we state that the operator
$[1+B^{\mathit{II}}(z)]^{-1}$ exists and it is bounded analytic in
$\Omega \setminus \mathcal{E}$.
  For $z\in \Omega \setminus
\mathcal{E}$ we construct $R^{\mathit{II}}_{\alpha , V_+, V_-
}(z)$ as
$$
R^{\mathit{II}}_{\alpha , V_+, V_- }(z) =
R^{\mathit{II}}_{\alpha}(z) - \check{R}^{\mathit{II}}_{\alpha }(z)
  \mathrm{V}^{1/2}\left[ 1+B^{\mathit{II}} (z)\right]^{-1}
|\mathrm{V} |^{1/2} \hat{R}^{\mathit{II}}_{\alpha }(z)\,,
$$
\cf~(\ref{eq-resolI}), where $\check{R}^{\mathit{II}}_{\alpha
}(z)\mathrm{V}^{1/2}$ and $\hat{R}^{\mathit{II}}_{\alpha
}(z)|\mathrm{V} |^{1/2}$ are the second sheet continuations
already discussed.
\end{proof}
\begin{corollary} \label{cor-absing}
$\sigma_\mathrm{sc}(H_{\alpha , V_+, V_- })\cap [\xi_0, \mu_0
]=\emptyset$.
\end{corollary}
\begin{proof}
It follows from the previous theorem that
$\sigma_\mathrm{sc}(H_{\alpha , V_+, V_- })\cap [\xi_0,\mu_0
]\subset \mathcal{E} $ is a finite set; this implies the claim.
\end{proof}
Without a danger of confusion, we employ the notation $\lambda
\mapsto E(\lambda )$ for the spectral resolution of $H_{\alpha
,V_0} $ for $\lambda \in [\xi_0 , \infty )\setminus \mathcal{E}$.
%
Then the operator $R_{\alpha ,V_0} (z)$ admits the following
decomposition
\begin{equation}\label{eq-decomR}
R_{\alpha ,V_0} (z) =\sum_{i=1}^N \frac{1}{\nu _i -z}P_{i} +\int
_{[\xi_0 , \infty)
} \frac{\mathrm{d}E (\lambda )}{\lambda -z }\,,
\end{equation}
where $
  \{\nu_i\}_{i=1}^N = \sigma _{\mathrm{p} }(H_{\alpha ,V_0} )
$,
$P_i =\omega_i\,(\omega_i , \cdot)$, and $\omega_i $ are the
corresponding eigenvectors.
Due to the definition of $E(\cdot )$ and statement of
Corollary~\ref{cor-absing}, we conclude that $ (\xi_0 , \mu_0)\ni
\lambda \mapsto E(\lambda )$ project onto $\mathcal{H}_{ac}$.
Given $f, g \in L^2 $, let us denote by $F(\cdot )_{f,g}$ the
Radon--Nikodym derivative of $ (\xi_0 , \mu_0)\ni \lambda \mapsto
(f,E(\lambda )g)$. The limit $z=\gamma  +i\varepsilon $ for
$\varepsilon \to 0 $ with $\gamma \in (\xi _0 , \mu_0)\setminus
\mathcal{E}$ and $f\,,g \in C_0 (\R^2)$ takes the form
\begin{align} \nonumber
 \big(f,R_{\alpha, V_0}(\gamma +i \,0 )g \big) = &
\sum _{i=1}^{N} \frac{1}{\nu _i -\gamma }( f ,\omega _i)(\omega_i
, g ) +\mathcal{P}\!\!\int _{[\xi_0, \mu_0) } \frac{
F_{f,g}(\lambda )\,\mathrm{d}\lambda }{\lambda -\gamma }
\\
\label{eq-uplimit} & + \int _{\mu_0 }^\infty
\frac{\mathrm{d}(f,E(\lambda ) g) }{\lambda -\gamma }
 + \pi \, i F_{f,g}(\gamma ) \,,
\end{align}
where the symbol $\mathcal{P}$ denotes the principle value. By the
edge-of-the-wedge theorem (\cf~\cite{Ru}), we get
\begin{equation} \label{eq-edge}
\big(f, R_{\alpha, V_0} (\gamma +i \,0 )g\big) = \big(f,R_{\alpha,
V_0}^\mathit{II}(\gamma-i\, 0 )g \big) \,, \quad f\,,g\in C_0
(\R^2) \,,
\end{equation}
where $R_{\alpha, V_0}^\mathit{II} (\cdot )$ denotes the second
sheet continuation of $R_{\alpha, V_0} (\cdot )$ stated in
Lemma~\ref{le-acspectrum}.
%
%
%
%
The operator $R_{\alpha, V_0} (\cdot )$ stands for resolvent of
the mirror symmetry system. Now we introduce the potential
$V_\epsilon =\epsilon V_p$ living on $\R\times \{+a\}$. By means
of $R^\mathit{II}_{\alpha,V_0}(z)$  we determine
$\Gamma^\mathit{II}(z)$ given by
\begin{equation}\label{eq-GammaII}
\Gamma ^\mathit{II} (z)=1
+|V_\epsilon|^{1/2}\mathrm{R}^\mathit{II}_{\alpha,
V_0}(z)V_\epsilon^{1/2}\,,\quad z\in \Omega\setminus
\mathcal{E}\,,
\end{equation}
where the second component of the above expression states the
analytic second sheet continuation of
$|V_\epsilon|^{1/2}\mathrm{R}_{\alpha,
V_0}(z)V_\epsilon^{1/2}\,:\,L^2_+ \to L^2_+$.

\subsection{Zeros of $\Gamma^\mathit{II}(\cdot)$; the Fermi golden rule}
Henceforth  we assume that the set of embedded eigenvalues of $
H_{\alpha ,V_0}$ is not empty (this is true, for instance, under
the hypotheses of Corollary~\ref{th-emb}). Then there exists an
integer $k_0 \leq N$ such that for all $i \geq k_0 $ we have
$\nu_i \in (\xi_0 ,\mu_0) $. Given $k\geq k_0$, define
\begin{equation}\label{eq-A}
A_k(z) := R_{\alpha , V_0}(z)-\frac{1}{\nu _k -z}P_{k}
\end{equation}
Analogously we define $A_k ^\mathit{II}(z)$ substituting
$R_{\alpha , V_0}(z)$ in (\ref{eq-A}) by $R^{\mathit{II}}_{\alpha
, V_0}(z)$.

The main  results of this section is contained in the following
theorem.
%
\begin{theorem}
Suppose~\eqref{eq-assumptionV}. Assume that
 the number $\nu_k \in (\xi_0 , \mu_0)$ is an embedded
eigenvalue of $H_{\alpha, V_0}$. Then the resolvent of $H_{\alpha
; \epsilon }$ has a pole~$z_k$ satisfying
$$
z_k=\nu_k +\epsilon \, ( \omega _k , V_p\omega _k )_{L^2_+} +
\epsilon^2 \;\! (\Gamma_r +i\Gamma_i) +\mathcal{O}(\epsilon^3 )
  \qquad \mbox{as} \qquad
  \epsilon \to 0
  \,,
$$
where
\begin{align*}
  \Gamma_r &:= -\sum _{i \in \{1,..., N\}\setminus k}
 \frac{1}{\nu _i -\nu_k}|( \omega_k
, V_p \omega _i)_{L^2_+}|^2 -\mathcal{P}\!\!\int _{[\xi_0, \infty
) } \frac{  \mathrm{d}\, \big(  \omega _k ,V_p E(\lambda )V_p
\omega_k\big)_{L^{2}_+} }{\lambda -\nu_k }
  \,,
  \\
  \Gamma_i &:= - \pi F_+(\nu_k )
  \,, \qquad \mbox{with} \qquad
  F_+(\nu _k):= \frac{\mathrm{d}}{\mathrm{d} \nu}
  \big( \omega _k ,V_p  E(\nu )V_p \omega_k\big)_{L^{2}_+} |_{\nu =\nu_k}
  \,.
\end{align*}
(Here the functions $\omega _k $ are understood as embeddings to
$L^2_+$. Similarly, $E(\cdot )$ is a family of operators acting
from $L^2_+$ to $L^2_+$.)
\end{theorem}
\begin{proof}
Note that the function $z\mapsto (f, A^\mathit{II}_k (z)g)$,
$f\,,g\in C_0 (\R^2 )$ is analytic in a neighbourhood $M$ of
$\nu_k$. Furthermore, for $\epsilon$ small enough, say $\epsilon
\leq \epsilon_0$, where $\epsilon_0 >0$ the operator $
  G_{\epsilon, k}  (z):=
  1+|V_\epsilon
  |^{1/2} A^\mathit{II}_{k} (z) V_\epsilon ^{1/2}
$ is invertible.
We define the function $\eta_k \,:\, [0, \epsilon_0 )\times M \to
\C$ by
$$
\eta_k (\epsilon , z) :=z-\nu_k -\big(\omega _k,
 V_{\epsilon
}^{1/2 } G_{\epsilon , k} (z)^{-1}|V_\epsilon |^{1/2} \omega _k
\big)_{L^2_+}\,.
$$
Suppose $z\in M\setminus \nu_k$. A straightforward calculation
using~(\ref{eq-GammaII}) and~(\ref{eq-A}) shows that
 $ \phi\in \mathrm{Ker}\,
\Gamma^{\mathit{II}}(z) $ if, and only if,
$$
\phi+\frac{1}{\nu_k -z} (\omega_k , V_\epsilon^{1/2}\phi )
G_{\epsilon ,k} (z)^{-1}|V_\epsilon|^{1/2}\omega_k =0\,.
$$
This means that $ \mathrm{Ker}\, \Gamma^{\mathit{II}}(z)\neq \{0\}
$ if, and only if, $z$~is a solution of
\begin{equation}\label{eq-expandz}
\eta_k(\epsilon , z) =0\,.
\end{equation}
After expanding $G_{\epsilon , k} (z)^{-1}$ with respect
to~$\epsilon$, function $\eta (\epsilon, z)$ reads as
\begin{equation}\label{eq-poleI}
\eta_k (\epsilon, z )= z-\nu_k -\epsilon \, (\omega_k , V_p
\omega_k )_{L^2_+} +\epsilon^2 \;\! \big( \omega_k,  V_p  A_k
^\mathit{II}(z) V_p \omega_k \big)_{L^2_+}+ \mathcal{O}(\epsilon^3
)\,.
\end{equation}
The function $\eta _k $ is analytic in $z\in M$ and it is $C^1$ in
both variables. It is clear that $\eta _k(0, \nu_k)=0$ and
$\partial_z \eta_k (0 , \nu_k) =1$. Using the implicit function
theorem to~(\ref{eq-expandz}) and applying~(\ref{eq-poleI}), we
state that there exists an open neighbourhood $\mathcal{U}_0\subset
\R_+$ of zero and the unique function $z_{k}\,:\, \mathcal{U}_0
\to \C $ of $\epsilon$ given by
$$
z_k =\nu_k +\epsilon \, (\omega_k , V_p \omega_k )_{L^2 _+}-
\epsilon^2 \;\! ( \omega_k ,  V_p A_k ^\mathit{II}( \nu_k -i \,0)
V_p \omega_k )_{L^2 _+}+ \mathcal{O}(\epsilon^3 )\,,
$$
being a zero of $\Gamma^\mathit{II}(\cdot)$.
Employing~(\ref{eq-A}), (\ref{eq-uplimit}) and~(\ref{eq-edge}), we
get the statement of the theorem.
\end{proof}
%

The above results can be summarized as follows. Suppose the mirror
symmetry system~\eqref{mirror} has an embedded eigenvalue $\nu$.
Once we break the symmetry introducing the ``perturbant"
$V_{\epsilon}= \epsilon V_p$, the pole of the resolvent shifts
from the spectrum and makes the second sheet continuation pole of
the resolvent. The imaginary component of pole is related to the
resonance width $\Gamma_w := -2 \Im \,z_k
=\mathcal{O}(\epsilon^2)$. This means that for $\epsilon$ small
the resonance is physically essential.
Employing the lowest order perturbation in~$\epsilon$, we can write
$$
\Gamma _w \ \approx \
2\pi \frac{\mathrm{d} \big(\omega_k ,V_p \mathrm{E}(\lambda ) V_p
\omega_k \big)_{L^2_+}}{\mathrm{d}\lambda } \Big|_{\lambda =\nu_k}
\epsilon^2 \,.
$$
This gives  the Fermi golden rule. Moreover, $\Gamma_w ^{-1}$
determines the life time of the resonance state.
\begin{remark} The phenomena of resonances induced by broken symmetry
was studied in~\cite{EK3,K12}. However, in the models analyzed
there the singular potentials are constants and the broken
symmetry has rather a geometrical character. The methods derived
in this paper  essentially differ from the technics applied in
\cite{EK3,K12}.
Finally, let us mention that the resonances phenomena and the
decay law were recently reviewed in \cite{Ex-resonances}. The
systems with singular potentials were considered as examples
of solvable models.
\end{remark}
\subsection*{Acknowledgement}
The authors are grateful to the anonymous referee
for a careful reading of the manuscript,
for pointing out an error contained in the first version of the paper
and for many other valuable remarks and suggestions.
The work has been partially supported by RVO61389005 and the GACR
grant No.\ P203/11/0701.

%

\providecommand{\bysame}{\leavevmode\hbox
to3em{\hrulefill}\thinspace}
\providecommand{\MR}{\relax\ifhmode\unskip\space\fi MR }
\providecommand{\MRhref}[2]{%
  \href{http://www.ams.org/mathscinet-getitem?mr=#1}{#2}
} \providecommand{\href}[2]{#2}

\end{document}